\documentclass[a4paper, 10pt, american]{amsart}

\usepackage{times,latexsym,amssymb}
\usepackage{amsmath,amsthm,bm}
\usepackage{color}
\usepackage[colorlinks,pdfpagelabels,pdfstartview = FitH,bookmarksopen
= true,bookmarksnumbered = true,linkcolor = blue,plainpages =
false,hypertexnames = false,citecolor = red,pagebackref=false]{hyperref}

\usepackage{amsbsy}
\usepackage{amstext}
\usepackage{amssymb}
\usepackage{esint}
\usepackage{stmaryrd}
\allowdisplaybreaks
\newtheorem{theorem}{Theorem}
\newtheorem{lemma}[theorem]{Lemma}
\newtheorem{definition}[theorem]{Definition}
\newtheorem{proposition}[theorem]{Proposition}

\numberwithin{theorem}{section}
\numberwithin{equation}{section}

\newcommand{\mint}{- \mskip-19,5mu \int}

\def\N{\mathbb{N}}
\def\R{\mathbb{R}}

\renewcommand{\d}{\mathrm{d}}
\newcommand{\dx}{\mathrm{d}x}

\newcommand{\dt}{\mathrm{d}t}
\newcommand{\ds}{\mathrm{d}s}
\newcommand{\dtau}{\mathrm{d}\tau}

\renewcommand{\epsilon}{\varepsilon}

\def\pf{\boldsymbol{\mathfrak p}}
\def\p{\beta}

\DeclareMathOperator{\Div}{div}

\renewcommand{\epsilon}{\varepsilon}
\newcommand{\eps}{\varepsilon}
\renewcommand{\rho}{\varrho}

\def\eqn#1$$#2$${\begin{equation}\label#1#2\end{equation}}

\newcommand{\power}[2]{\bm{#1^{\mbox{\unboldmath{\scriptsize$#2$}}}}}
\newcommand{\abs}[1]{|#1|}
\newcommand{\babs}[1]{\big|#1\big|}

\def\Xint#1{\mathchoice
    {\XXint\displaystyle\textstyle{#1}}%
    {\XXint\textstyle\scriptstyle{#1}}%
    {\XXint\scriptstyle\scriptscriptstyle{#1}}%
    {\XXint\scriptscriptstyle\scriptscriptstyle{#1}}%
    \!\int}
\def\XXint#1#2#3{\setbox0=\hbox{$#1{#2#3}{\int}$}
    \vcenter{\hbox{$#2#3$}}\kern-0.5\wd0}
\def\bint{\Xint-}
\def\dashint{\Xint{\raise4pt\hbox to7pt{\hrulefill}}}

\def\Xiint#1{\mathchoice
    {\XXiint\displaystyle\textstyle{#1}}%
    {\XXiint\textstyle\scriptstyle{#1}}%
    {\XXiint\scriptstyle\scriptscriptstyle{#1}}%
    {\XXiint\scriptscriptstyle\scriptscriptstyle{#1}}%
    \!\iint}
\def\XXiint#1#2#3{\setbox0=\hbox{$#1{#2#3}{\iint}$}
    \vcenter{\hbox{$#2#3$}}\kern-0.5\wd0}
\def\biint{\Xiint{-\!-}}

\subjclass[2010]{35B65, 35K40, 35K55}

\begin{document}
\title[Higher integrability for doubly nonlinear parabolic systems]{Higher integrability for\\ doubly nonlinear parabolic systems}
\date{\today}

\author[V. B\"ogelein]{Verena B\"{o}gelein}
\address{Verena B\"ogelein\\
Fachbereich Mathematik, Universit\"at Salzburg\\
Hellbrunner Str. 34, 5020 Salzburg, Austria}
\email{verena.boegelein@sbg.ac.at}

\author[F. Duzaar]{Frank Duzaar}
\address{Frank Duzaar\\
Department Mathematik, Universit\"at Erlangen--N\"urnberg\\
Cauerstrasse 11, 91058 Erlangen, Germany}
\email{duzaar@math.fau.de}

\author[J. Kinnunen]{Juha Kinnunen}
\address{Juha Kinnunen\\
Aalto University, Department of Mathematics and Systems Analysis\\
P.O. Box 11100, FI-00076 Aalto, Finland}
\email{juha.k.kinnunen@aalto.fi}

\author[C. Scheven]{Christoph Scheven}
\address{Christoph Scheven\\ Fakult\"at f\"ur Mathematik, 
Universit\"at Duisburg-Essen\\45117 Essen, Germany}
\email{christoph.scheven@uni-due.de}

\maketitle


\begin{abstract}
This paper proves a local higher integrability result for the spatial gradient of weak solutions to
doubly nonlinear parabolic systems. 
The new feature of the argument is that the intrinsic geometry involves the solution as well as its spatial gradient.
The main result holds true for a range of parameters suggested by other nonlinear parabolic systems.
\end{abstract}

\section{Introduction}
This paper studies regularity of the spatial gradient of weak solutions to doubly nonlinear parabolic equations (systems) of the type
\begin{equation}\label{eq:Trudinger}
	\partial_t \big(|u|^{p-2}u\big) -\Div \big(|Du|^{p-2}Du\big) 
	= 
	\Div \big(|F|^{p-2}F\big) 
\end{equation}
with $1<p<\infty$ in a space-time cylinder 
$\Omega_T:= \Omega\times (0,T)$, where $\Omega\subset\R^n$ is a bounded  domain,
$n\ge 1$, and $T>0$.
Equation \eqref{eq:Trudinger} is a special case of the general doubly nonlinear parabolic equation 
\begin{equation}\label{eq:doubly-gen}
	\partial_t \big(|u|^{m-1}u\big) -\Div \big(|Du|^{p-2}Du\big) 
	= 
	\Div \big(|F|^{p-2}F\big) ,
\end{equation}
with $p>1$ and $m>0$. 
This includes the parabolic $p$-Laplacian and the porous medium equation. Note that with the choice $m=p-1$  we recover \eqref{eq:Trudinger}.  
Equation \eqref{eq:doubly-gen} has a different behavior when $m<p-1$ and $m\ge p-1$.
The first range is called the slow diffusion case, since disturbances propagate with a finite speed and free boundaries occur, while
in the second range disturbances propagate with infinite speed and extinction in finite time may occur.
This is called the fast diffusion case.
In this sense, equation \eqref{eq:Trudinger} represents the borderline case between the slow and fast diffusion ranges.

One might expect that the regularity theory for the doubly nonlinear equation~\eqref{eq:Trudinger} is similar to the one for the heat equation. 
In fact, the equation is homogeneous, in the sense that solutions are invariant under multiplication by constants. 
In addition, a scale and location invariant parabolic Harnack's inequality holds true for non-negative weak solutions, see \cite{Trudinger,Kinnunen-Kuusi}. However, in this case Harnack's inequality does not immediately imply H\"older continuity of solutions, which indicates that there is a difference compared to the heat equation. 
The main difficulty with~\eqref{eq:Trudinger} is that adding a constant to a solution destroys the property of being a solution. 
The general doubly nonlinear equation \eqref{eq:doubly-gen} is non-homogeneous and an intrinsic geometry is used in the regularity theory, 
i.e.~the space-time scaling of cylinders depends either on the solution or the spatial gradient of the solution.
The idea that the inhomogeneous behavior of a nonlinear parabolic equation can be compensated by an intrinsic geometry goes back to the pioneering work of DiBenedetto and Friedman, see for example the monograph \cite{DiBe}.
The regularity theory of weak solutions of \eqref{eq:Trudinger} and \eqref{eq:doubly-gen} is reasonably developed, at least in the scalar case for non-negative solutions; see \cite{Trudinger, Gianazza-Vespri, Kinnunen-Kuusi, DBGV-book} for Harnack's inequality, \cite{Vespri,Kuusi-et-al:1,Kuusi-et-al:2} for H\"older  regularity results, and finally \cite{Siljander} for Lipschitz regularity with respect to the spatial variable for solutions bounded from below by a positive constant. 
However, little is known about signed solutions, regularity of the gradient of a weak solution and systems. 

The primary purpose of this paper is to establish a local higher
integrability result for the spatial gradient of weak solutions to
parabolic equations and systems of the type \eqref{eq:Trudinger}.
We show that there exists a constant $\varepsilon>0$, such that
$$
	|Du|^{p(1+\epsilon)}\in L^1_{\rm loc}(\Omega_T),
$$
whenever $u$ is a weak solution to the equation or the system.
In particular, our result ensures
that weak solutions of \eqref{eq:Trudinger} belong to a slightly
better Sobolev space than the natural energy space and therefore obey a self-improving property of integrability.
Our result comes with a reverse H\"older type estimate, see Theorem~\ref{thm:higherint}.
The higher integrability for the doubly nonlinear equation~\eqref{eq:Trudinger} has been an open problem for a long time.
Here we give an answer to this question in the range
\[
  \max\big\{\tfrac{2n}{n+2},1\big\}<p< \tfrac{2n}{(n-2)_+}.
\]
This range may seem unexpected, but the lower bound also appears in the higher integrability for the parabolic $p$-Laplace system \cite{Kinnunen-Lewis:1}, while the upper bound is exactly the expected one for the porous medium system in the fast diffusion range.
For $n=1$ and $n=2$ our result applies whenever $1<p<\infty$.
It remains an open question whether the corresponding result holds true when $n\ge3$. 

The key ingredient in the proof  of our main result is a suitable intrinsic geometry. 
By now, variants of this idea have been successfully used in establishing the higher integrability for the parabolic $p$-Laplace system \cite{Kinnunen-Lewis:1} and very recently for the porous medium equation \cite{Gianazza-Schwarzacher} and system \cite{BDKS-higher-int}. 
Our idea is to consider space-time cylinders $Q_{r,s}(z_o):=B_r(x_o)\times(t_o-s,t_o+s)$, with $z_o=(x_o,t_o)$, such that the quotient $\frac{s}{r^p}$ satisfies
\begin{equation}\label{geometry}
	\frac{s}{r^p} = \mu^{p-2}
	\quad\mbox{with}\quad
	\mu^p
	\approx
	\frac{\displaystyle{\biint_{Q_{r,s}(z_o)} 
	\frac{\abs{u}^{p}}{r^p}\dx\dt}}
	{\displaystyle{\biint_{Q_{r,s}(z_o)} \big[ |Du|^p +|F|^p\big]\dx\dt}}\,.
\end{equation}
This geometry involves the solution as well as its spatial gradient and therefore allows to balance the mismatch between $|u|$ and $|Du|$ in the equation. To our knowlegde this is the first time that such a geometry is used. On these cylinders we are able to prove Sobolev-Poincar\'e and reverse H\"older type inqualites. The construction of the cylinders is quite involved, since the cylinders on the right-hand side of \eqref{geometry} also depend on the parameter $\mu$. 
In the course of the construction we modify the argument in \cite{Gianazza-Schwarzacher}; see also \cite{BDKS-higher-int}.

In the stationary elliptic case 
the higher integrability was first observed by Elcrat \& Meyers \cite{Meyers-Elcrat},
see also the monographs \cite[Chapter 11, Theorem 1.2]{Giaquinta:book} and \cite[Section 6.5]{Giusti:book}.
The first higher integrability result, in the context of parabolic systems, can be found in
\cite[Theorem 2.1]{Giaquinta-Struwe}. 
The higher integrability for the gradient of solutions for general parabolic systems with $p$-growth has been established by Kinnunen \& Lewis \cite{Kinnunen-Lewis:1}. 
This local interior result has been generalized in the meantime in various directions, e.g.~ 
global results, higher order parabolic systems (interior and at the boundary); 
see \cite{Parviainen, Boegelein:1, Boegelein-Parviainen}.
For the  porous medium equation, i.e.~equation \eqref{eq:doubly-gen}
with $p=2$, the question of higher integrability turned out to be  more
challenging than for the parabolic $p$-Laplace equation, i.e.~equation \eqref{eq:doubly-gen}
with $m=1$. 
The problem was solved only recently by Gianazza \& Schwarzacher \cite{Gianazza-Schwarzacher}. They proved that non-negative weak solutions to the  porous medium equation
possess the higher integrability for the
spatial gradient. 
Their proof, however, uses the method of expansion of positivity
and therefore cannot be extended to signed solutions and porous medium
type systems. A simpler and more flexible proof, which does not rely
on the expansion of positivity, is given in \cite{BDKS-higher-int},
where higher integrability for porous medium type systems is
achieved. As special case, signed solutions are included in this result.

\section{Notation and the main result}

\subsection{Notation}
Throughout the paper we use space-time cylinders of the form
\begin{equation}\label{cylinder}
	Q_\rho^{(\mu)}(z_o)
	:=
	\left\{\begin{array}{ll}
	B_\rho(x_o)\times \Lambda_\rho^{(\mu)}(t_o),&
	\quad\mbox{if $p<2$},\\[5pt]
	B_\rho^{(\mu)}(x_o)\times \Lambda_\rho(t_o),&
	\quad\mbox{if $p\ge2$},
	\end{array}\right.
\end{equation}
with center $z_o=(x_o,t_o)\in \R^n\times\R$, radius $\rho>0$ and scaling  parameter $\mu >0$, where 
$$
	B_\rho^{(\mu)}(x_o)
	:=
	\big\{x\in\R^n: |x-x_o|< \mu^{\frac{2-p}{p}}\rho\big\},
	\quad
	B_\rho(x_o)
	:=
	B_\rho^{(1)}(x_o).
$$ 
and
$$
	\Lambda_\rho^{(\mu)}(t_o)
	:=
	\big(t_o-\mu^{p-2}\rho^{p}, 
	t_o+\mu^{p-2}\rho^{p}\big),
	\quad
	\Lambda_\rho(t_o)
	:=
	\Lambda_\rho^{(1)}(t_o).
$$ 
Note that in both cases the cylinders \eqref{cylinder} admit the scaling property \eqref{geometry}$_1$. Moreover, they satisfy the inclusion 
$$
	Q_\rho^{(\mu_2)}(z_o)
	\subseteq
	Q_\rho^{(\mu_1)}(z_o)
	\quad\mbox{whenever }
	\mu_1\le\mu_2.
$$
In the case that $\mu =1$, we 
omit the scaling parameter in our notation and instead of $Q_\rho^{(1)}(z_o)$ we write $Q_\rho(z_o)$. 
For a map $u\in L^1\big(0,T;L^1(\Omega,\R^N)\big)$ and a given measurable set $A\subset\Omega$ 
with positive Lebesgue measure the slicewise mean $\langle u\rangle_{A}\colon (0,T)\to \R^N$ of $u$ on  $A$ is defined
by
\begin{equation*}
	\langle u\rangle_{A}(t)
	:=
	\mint_{A} u(\cdot,t)\,\dx,
	\quad\mbox{for a.e.~$t\in(0,T)$.}
\end{equation*}
Note that if $u\in C^0\big([0,T];L^p(\Omega,\R^N)\big)$ the slicewise means are defined for any $t\in [0,T]$.
If the set $A$ is a ball $B_\rho^{(\mu)}(x_o)$, then we abbreviate $\langle u\rangle_{x_o;\rho}^{(\mu)}(t):=\langle u\rangle_{B_\rho^{(\mu)}(x_o)}(t)$ and $\langle u\rangle_{x_o;\rho}(t):=\langle u\rangle_{x_o;\rho}^{(1)}(t)$ for $\mu=1$.
Similarly, for a given measurable set $E\subset\Omega_T$ of positive Lebesgue
measure  the mean value $(u)_{E}\in \R^N$ of $u$ on  $E$ is defined by
\begin{equation*}
	(u)_{E}
	:=
	\biint_{E} u\,\dx\dt.
\end{equation*}
If $E\equiv Q_\rho^{(\mu)}(z_o)$, we abbreviate
$(u)^{(\mu)}_{z_o;\rho}:=(u)_{Q_\rho^{(\mu)}(z_o)}$. Moreover, we often write $u(t):=u(\cdot,t)$ for notational
convenience.
For the power of a vector $u\in\R^N$, we use the short-hand notation 
\begin{equation*}
  \power{u}{\alpha}:=|u|^{\alpha-1}u,
  \qquad\mbox{for $\alpha>0$},
\end{equation*}
which we interpret as $\power{u}{\alpha}=0$ in the case $u=0$ and
$\alpha\in(0,1)$. 
Finally, we let $\pf:=\max\{p,2\}$. 

\subsection{Assumptions and the main result}
We consider general systems of the  type
\begin{equation}\label{eq-doubly}
	\partial_t \big(|u|^{p-2}u\big) -
	\Div \mathbf A (x,t,u,Du) 
	=
	\Div \big(|F|^{p-2}F\big)\quad\mbox{in $\Omega_T$}
\end{equation}
where the vector-field $\mathbf A\colon \Omega_T\times\R^N\times\R^{Nn}\to \R^{Nn}$ is a Carath\'eodory function satisfying  
the standard $p$-growth and coercivity conditions
\begin{equation}\label{growth-a}
\left\{
\begin{array}{c}
	\mathbf A(x,t,u,\xi)\cdot\xi\ge \nu |\xi|^p\, ,\\[6pt]
	| \mathbf A(x,t,u,\xi)|\le L |\xi|^{p-1},
\end{array}
\right.
\end{equation}
for a.e.~$z=(x,t)\in \Omega_T$ and any $(u,\xi)\in \R^N\times\R^{Nn}$, where $0<\nu\le L<\infty$ are positive constants. 
In order to formulate our main result, we need to introduce the concept of weak solution.

\begin{definition}\label{def:weak_solution}\upshape
Assume that the vector field
$\mathbf A\colon \Omega_T\times \R^N\times\R^{Nn}\to\R^{Nn}$ satisfies \eqref{growth-a}. 
A measurable function
$u\colon\Omega_T\to\R^N$ in the class
\begin{equation*}
	u\in C^0 \big([0,T]; L^{p}(\Omega;\R^N)\big) \cap 
	L^p\big(0,T;W^{1,p}(\Omega;\R^N)\big)
\end{equation*} 
is a \textit{weak solution} to the doubly non-linear parabolic system  \eqref{eq-doubly} if and only if the identity
\begin{align}\label{weak-solution}
	\iint_{\Omega_T}\big[|u|^{p-2}u\cdot\varphi_t - \mathbf A(x,t,u,Du)\cdot D\varphi\big]\dx\dt
    =
    \iint_{\Omega_T} |F|^{p-2}F\cdot D\varphi \,\dx\dt 
\end{align}
holds true, for every testing function $\varphi\in C_0^\infty(\Omega_T,\R^N)$.
\hfill$\Box$
\end{definition}

The following theorem is our main result.

\begin{theorem}\label{thm:higherint}
Let
\begin{equation}\label{p-bound}
  \max\big\{\tfrac{2n}{n+2},1\big\}<p< \tfrac{2n}{(n-2)_+},
\end{equation}
where the right-hand side is interpreted as $\infty$ for the dimensions $n=1$
and $n=2$, and assume that $\sigma>p$. Then, there exists $\eps_o=\eps_o(n,p,\nu,L)\in (0,1]$ such that whenever 
$F\in L^\sigma(\Omega_T,\R^N)$ and 
$u$ is a weak solution to \eqref{eq-doubly} in the sense of Definition~{\upshape\ref{def:weak_solution}}, then there holds
$$
  	D u
	\in 
	L^{p(1+\eps_1)}_{\rm loc}\big(\Omega_T,\R^{Nn}\big),
$$
where $\epsilon_1:=\min\big\{\eps_o,\frac{\sigma}{p}-1\big\}$. 
Moreover, for every $\eps\in(0,\epsilon_1]$ and every cylinder
$
  Q_{2R}\subseteq\Omega_T
$,
we have
\begin{align*}
	\biint_{Q_{R}} |Du|^{(1+\epsilon)p} \,\dx\dt 
	&\le
	c \Bigg[1+\biint_{Q_{2R}} 
	\bigg[\frac{\abs{u}^{p}}{(2R)^p} + |Du|^p\bigg] \dx\dt
	\Bigg]^{\epsilon}
	\biint_{Q_{2R}} |Du|^{p} \,\dx\dt \\
	&\quad+
	c\,\biint_{Q_{2R}} |F|^{(1+\epsilon)p} \,\dx\dt,
\end{align*}
where $c=c(n,p,\nu,L)$. \hfill $\Box$
\end{theorem}

Although Theorem~\ref{thm:higherint} is proved for exponents $p$
in the range \eqref{p-bound}, we indicate in each sub-step of the
proof what are the exact restrictions on $p$ that are needed in the
particular step.
In this way, the reader can easily retrace where restriction \eqref{p-bound} occurs.

\section{Auxiliary Material}

In order to ``re-absorb'' certain terms, we will use the following iteration lemma, cf.~\cite[Lemma 6.1]{Giusti:book}.

\begin{lemma}\label{lem:tech}
Let $0<\vartheta<1$, $A,C\ge 0$ and $\alpha > 0$. Then there exists a constant $c = c(\alpha,\vartheta)$
such that whenever  $\phi\colon[r,\rho]\to [0,\infty)$, with $0<r<\rho$, is a non-negative bounded function satisfying
\begin{equation*}
	\phi(t)
	\le
	\vartheta\, \phi(s) + \frac{A}{(s-t)^\alpha} + C
	\qquad \text{for all $r\le t<s\le \varrho$,}
\end{equation*}
then
\begin{equation*}
	\phi(r)
	\le
	c\,  \bigg[\frac{A}{(\varrho - r)^\alpha} + C\bigg].
\end{equation*}
\end{lemma}

The next lemma can be deduced as in \cite[Lemma~8.3]{Giusti:book}.

\begin{lemma}\label{lem:Acerbi-Fusco}
For any $\alpha>0$, there exists a constant $c=c(\alpha)$ such that,
for all $a,b\in\R^N$, $N\in\N$, we have
\begin{align*}
	\tfrac1c\big|\power{b}{\alpha} - \power{a}{\alpha}\big|
	\le
	\big(|a| + |b|\big)^{\alpha-1}|b-a|
	\le
	c \big|\power{b}{\alpha} - \power{a}{\alpha}\big|.
\end{align*}
\end{lemma}

The following lemma is a simple consequence of Lemma~\ref{lem:Acerbi-Fusco}.

\begin{lemma}\label{lem:a-b}
For any $\alpha\ge 1$, there exists a constant $c=c(\alpha)$ such that,
for all $a,b\in\R^N$, $N\in\N$, we have
\begin{align*}
	|b-a|^\alpha
    \le
    c\big|\power{b}{\alpha} - \power{a}{\alpha}\big|.
\end{align*}
\end{lemma}

The next
lemma provides useful estimates for the boundary term 
\begin{align}\label{def:b}
	\mathfrak b[u,v]
	:=
	\tfrac{1}{p}|v|^{p} - \tfrac{1}{p}|u|^{p} - \power{u}{p{-}1}\cdot(v-u),
	\quad\mbox{for $u,v\in\R^N$.}
\end{align}

\begin{lemma}\label{lem:basic-est-b}
For any $p\ge 1$ there exists a constant $c=c(p)$ such that for any $u,v\in\R^N$, $N\in\N$, we have
\begin{align*}
   \tfrac1c\, \mathfrak b[u,v]
   \le
   \big|\power{u}{\frac{p}{2}}-\power{v}{\frac{p}{2}}\big|^2 
   \le
   c\,\mathfrak b[u,v].
\end{align*}
\end{lemma}

\begin{proof} The case $1<p\le 2$
  follows from \cite[Lemma~2.3\,(i)]{BDKS-higher-int} applied with
  $m=\frac{1}{p-1}$. 
  Therefore it remains to consider the case
  $p>2$.
  In the following we denote $\phi_p(u):=\frac{1}{p}|u|^{p}$. 
With the abbreviations $a=\power{u}{p{-}1}$ and $b=\power{v}{p{-}1}$ we compute
\begin{align*}
	\mathfrak b[u,v]
	&=
	\phi_p(v)-\phi_p(u)- D\phi_p(u)(v-u) \\
        &=
	\tfrac{1}{p}|b|^{\frac{p}{p-1}} - 
	\tfrac{1}{p}|a|^{\frac{p}{p-1}} -
	a\cdot\big(\power{b}{\frac{1}{p-1}} -
        \power{a}{\frac{1}{p-1}}\big)
      \\
      &=
	\tfrac{p-1}{p}|a|^{\frac{p}{p-1}} - 
	\tfrac{p-1}{p}|b|^{\frac{p}{p-1}} -
	\power{b}{\frac{1}{p-1}}\cdot (a-b)
      \\
	&=
	\phi_{\frac{p}{p-1}}(a)-
	\phi_{\frac{p}{p-1}}(b)- 
	D\phi_{\frac{p}{p-1}}(b)(a-b).
\end{align*}
Since $\frac{p}{p-1}<2$ we may apply Lemma~\ref{lem:basic-est-b} in the subquadratic case. In this way we obtain
\begin{align*}
	\mathfrak b[u,v]
	\le
	c\big|\power{a}{\frac{p}{2(p-1)}}-
	\power{b}{\frac{p}{2(p-1)}}\big|^2 
	=
	c\big|\power{u}{\frac{p}{2}}-\power{v}{\frac{p}{2}}\big|^2 
\end{align*}
and 
\begin{align*}
	\mathfrak b[u,v]
	\ge
	\tfrac1c\big|\power{a}{\frac{p}{2(p-1)}}-
	\power{b}{\frac{p}{2(p-1)}}\big|^2 
	=
	\tfrac1c\big|\power{u}{\frac{p}{2}}-\power{v}{\frac{p}{2}}\big|^2 .
\end{align*}
This finishes the proof of the lemma.
\end{proof}

It is well known that mean values over subsets $A\subset B$ are quasi-minimizers in the integral $a\mapsto \int_B |u-a|^p \dx$. 
The following statement shows that mean values over subsets are still quasi-minimizing for $u^\alpha$ with $\alpha\ge\frac1p$. 
For $p=2$ and $A=B$, the lemma has been proved in \cite[Lemma 6.2]{Diening-Kaplicky-Schwarzacher}; see also \cite[Lemma 2.6]{BDKS-higher-int}. 
Here, we state a general version for powers. As expected, the quasi-minimality constant depends on the ratio of the measures of the set and the subset.

\begin{lemma} \label{lem:alphalemma}
Let $p\ge 1$ and $\alpha\ge\frac1p$. Then, there exists a constant 
$c=c(\alpha,p)$ such that whenever 
$A\subseteq B\subset \R^k$, $k\in\N$, are two bounded domains of positive measure, then for any function $u \in L^{\alpha p}(B,\R^N)$ and any constant $a\in\R^N$, we have 
$$
	\mint_B \big|\power{u}{\alpha}-\power{(u)_A}{\alpha}\big|^p \dx 
	\leq 
	\frac{c\,|B|}{|A|} \mint_B \big|\power{u}{\alpha}-\power{a}{\alpha}\big|^p \dx.
$$
\end{lemma}

\begin{proof}
The key step in the proof is the estimate of the difference $|\power{(u)_A}{\alpha}-\power{a}{\alpha}|$. 
In the {\bf case} $\boldsymbol{\alpha\ge1}$, we use
Lemmas~\ref{lem:Acerbi-Fusco} and~\ref{lem:a-b} in order to
obtain for a constant $c=c(\alpha,p)$ that 
\begin{align}\label{L26:alpha>1}
	\big|\power{(u)_A}{\alpha}-\power{a}{\alpha}\big|^p
	&\le
	c\big[|(u)_A|^{(\alpha-1)p}+|a|^{(\alpha-1)p}\big]
	|(u)_A - a|^p \nonumber\\
	\nonumber
	&\le
	c\big[|(u)_A-a|^{(\alpha-1)p}+|a|^{(\alpha-1)p}\big]
	|(u)_A-a|^p\\\nonumber
	&\le
        c\,|(u)_A-a|^{\alpha p}
        +
	c\,|a|^{(\alpha-1)p}\,\bint_A |u-a|^p \dx 
	\\\nonumber
	&\le
        c\, \bint_A|u-a|^{\alpha p} \dx
        +
	c\, \bint_A
        \big|\power{u}{\alpha}-\power{a}{\alpha}\big|^p \dx 
	 \\
	&\le
	c\,\bint_A \big|\power{u}{\alpha}-\power{a}{\alpha}\big|^p \dx .
\end{align}
Our next goal is to derive the same bound in the
{\bf case} $\boldsymbol {\frac1p\le\alpha<1}$. We begin by applying Lemma~\ref{lem:Acerbi-Fusco} to obtain
\begin{align}\label{L26:alpha<1-start}
    \big|\power{(u)_A}{\alpha}-\power{a}{\alpha}\big|
    &\le
    c\big[ 
    |(u)_A|+|a|\big]^{\alpha-1}|(u)_A -a| \nonumber\\
    &\le 
    c\, \bint_A\big[|(u)_A|+|a|\big]^{\alpha-1}|u-a|\,dx
 \end{align}
 and distinguish between two cases. In points $x\in A$ with
 $|u(x)|<2^{\frac1\alpha}|a|$, we use the elementary bound
\begin{align*}
     |(u)_A| + |a|
     \ge
     |a|
     \ge
     c(\alpha)\big[|u(x)|+|a|\big], 
\end{align*}
the fact $\alpha-1<0$ and Lemma \ref{lem:Acerbi-Fusco} in order to estimate the above integrand by
\begin{align}\label{L26:case1}
     \big[|(u)_A| + |a|\big]^{\alpha-1}|u(x)-a|
     &\le
     c\big[|u(x)|+|a|\big]^{\alpha-1}|u(x)-a| \nonumber\\
     &\le
     c\big|\power{u(x)}{\alpha}-\power{a}{\alpha}\big|.
\end{align}
In the remaining case $|u(x)|\ge 2^{\frac1\alpha}|a|$, we have
$\frac12|u(x)|^\alpha\le |u(x)|^\alpha-|a|^\alpha$, which we
use for the estimate
\begin{align*}
     |u(x)-a|^\alpha
     &\le
     2^\alpha|u(x)|^\alpha
     \le
     2^{\alpha+1}\big[|u(x)|^\alpha-|a|^\alpha\big]
     \le
     2^{\alpha+1}\big|\power{u(x)}{\alpha}-\power{a}{\alpha}\big|.
\end{align*}
We use this and the fact $[|(u)_A|+|a|]^\alpha\ge c|\power{(u)_A}{\alpha}-\power{a}{\alpha}|$ in
order to estimate the integrand in \eqref{L26:alpha<1-start} by 
\begin{equation}\label{L26:case2}
     \big[|(u)_A| + |a|\big]^{\alpha-1}|u(x)-a|
     \le
     c \big|\power{(u)_A}{\alpha}-\power{a}{\alpha}\big|^{\frac{\alpha-1}{\alpha}}
     \big|\power{u(x)}{\alpha}-\power{a}{\alpha}\big|^{\frac1\alpha}.
\end{equation}
Now we join the two cases. In view of \eqref{L26:case1} and
\eqref{L26:case2}, the estimate \eqref{L26:alpha<1-start} yields the
 bound
   \begin{align*}
     \big|\power{(u)_A}{\alpha}-\power{a}{\alpha}\big|^p
     \le
     c\,
     \bint_A|\power{u}{\alpha}-\power{a}{\alpha}|^p\,dx
     +
     c\big|\power{(u)_A}{\alpha}-\power{a}{\alpha}\big|^{\frac{(\alpha-1)p}{\alpha}}
     \bigg[\bint_A|\power{u}{\alpha}-\power{a}{\alpha}|^{\frac1\alpha}\,dx\bigg]^p. 
\end{align*}
We multiply this inequality by
$|\power{(u)_A}{\alpha}-\power{a}{\alpha}|^{\frac{(1-\alpha)p}\alpha}$,
apply Young's inequality with exponents $\frac1{1-\alpha}$,
$\frac1\alpha$ to the first term on the right-hand side,
and H\"older's inequality with exponents
$\alpha p,\frac{\alpha p}{\alpha p-1}$ to the second term. Note that both is
possible in the case $\frac1p<\alpha<1$, while in the case $\alpha=\frac1p$ the application of H\"older's inequality is not necessary. This procedure
results in the estimate   
\begin{align*}
     \big|\power{(u)_A}{\alpha}-\power{a}{\alpha}\big|^{\frac p\alpha} 
     &\le
     c\big|\power{(u)_A}{\alpha}-\power{a}{\alpha}\big|^{\frac{(1-\alpha)p}\alpha}
     \bint_A|\power{u}{\alpha}-\power{a}{\alpha}|^p\,dx
     +
     c\bigg[\bint_A|\power{u}{\alpha}-\power{a}{\alpha}|^{\frac1\alpha}\,dx\bigg]^p
     \\
     &\le\tfrac12\big|\power{(u)_A}{\alpha}-\power{a}{\alpha}\big|^{\frac{p}\alpha}
     +
     c\bigg[\bint_A|\power{u}{\alpha}-\power{a}{\alpha}|^p\,dx\bigg]^{\frac1\alpha}. 
\end{align*}
The second-last term can be re-absorbed into the left-hand side, which leads us to 
\begin{equation}\label{L26:alpha<1}
     \big|\power{(u)_A}{\alpha}-\power{a}{\alpha}\big|^p
     \le
     c\,\bint_A|\power{u}{\alpha}-\power{a}{\alpha}|^p\,dx.
\end{equation}
This is the estimate \eqref{L26:alpha>1} now also for the case
$\frac1p\le\alpha<1$. In any case, we can apply either \eqref{L26:alpha>1} or \eqref{L26:alpha<1} to conclude
\begin{align*}
	\mint_B \big|\power{u}{\alpha}-\power{(u)_A}{\alpha}\big|^p \dx 
	&\leq 
	2^{p-1} \mint_B \big|\power{u}{\alpha}-\power{a}{\alpha}\big|^p \dx +
	2^{p-1} \big|\power{(u)_A}{\alpha}-\power{a}{\alpha}\big|^p \\
	&\leq 
	2^{p-1} \mint_B \big|\power{u}{\alpha}-\power{a}{\alpha}\big|^p \dx +
	c \mint_A \big|\power{u}{\alpha}-\power{a}{\alpha}\big|^p \dx\\
	&\leq 
	\frac{c\,|B|}{|A|} \mint_B \big|\power{u}{\alpha}-\power{a}{\alpha}\big|^p \dx ,
\end{align*}
which proves  the claim.
\end{proof}

Finally, we state Gagliardo-Nirenberg's inequality in the form we will use in the sequel.

\begin{lemma}\label{lem:gag}
Let $1\le p,q,r<\infty$ and $\theta\in(0,1)$ such that
$ - \frac{n}{p}
\le \theta (1 - \frac{n}{q} ) - ( 1- \theta ) \frac{n}{r}$. Then there exists a constant
$c=c(n,p)$ such that for any ball $B_\rho(x_o)\subset\R^n$ with $\varrho>0$ and any function $u \in W^{1,q}(B_\rho(x_o))$, we have
\begin{equation*}
	\mint_{B_\rho(x_o)} \frac{|u|^p}{\rho^p} \dx
	\le
	c\,\Bigg[
	\mint_{B_\rho(x_o)} \bigg[
	\frac{|u|^q}{\rho^q} + |Du|^q \bigg] \dx
	\Bigg]^{\frac{\theta p}{q}}
	\bigg[\mint_{B_\varrho(x_o)} \frac{|u|^r}{\rho^r} \,\dx\bigg]
	^{\frac{(1-\theta)p}{r}} .
\end{equation*}
\end{lemma}

\section{Energy bounds}

In this section we exploit the doubly nonlinear system \eqref{eq-doubly} in order to deduce an energy estimate and a gluing lemma. These are the only points in the proof where the fact that $u$ is a solution of \eqref{eq-doubly} is used.

\begin{lemma}\label{lem:energy}
Let $p>1$ and u be a weak solution to \eqref{eq-doubly} in $\Omega_T$ in the sense of Definition~{\upshape\ref{def:weak_solution}}. 
Then, on any cylinder $Q_{R,S}(z_o):=B_R(x_o)\times\Lambda_S(t_o)\subseteq\Omega_T$ with
$R,S>0$, and for all $r\in[R/2,R)$, $s\in [S/2^p,S)$ and $a\in\R^N$,
we have
\begin{align*}
	\sup_{t \in \Lambda_s (t_o)} &
	\mint_{B_r (x_o)} 
	\frac{|\power{u}{\frac{p}{2}}(t)-\power{a}{\frac{p}{2}}|^2}{S} \,\dx +
	\biint_{Q_{r,s}(z_o)} |Du|^p \,\dx\dt \\
	&\leq 
	c\,
	\biint_{Q_{R,S}(z_o)} 
	\bigg[ 
	\frac{|\power{u}{\frac{p}{2}}-\power{a}{\frac{p}{2}}|^2}{S-s} + 
	\frac{|u-a|^p}{(R-r)^p} + |F|^p\bigg] \dx\dt,
\end{align*}
where $c=c(p,\nu,L)$. 
\end{lemma}

\begin{proof}[Proof.]
For $v\in L^1(\Omega_T,\R^N)$, we define the following mollification in time 
\begin{equation*}
	\llbracket v \rrbracket_h(x,t)
	:= 
	\tfrac 1h \int_0^t \mathrm e^{\frac{s-t}h} v(x,s) \, \ds.
\end{equation*}
From the weak form \eqref{weak-solution} of the differential equation we deduce the mollified version 
\begin{align}\label{mol-eq}
	\iint_{\Omega_T} &
	\Big[\partial_t \llbracket \power{u}{p{-}1}\rrbracket_h\cdot\varphi + 
	\llbracket\mathbf A(x,t,u,Du)\rrbracket_h\cdot D\varphi \Big]\dx\dt 
	\nonumber\\
	&=
	\iint_{\Omega_T} \llbracket |F|^{p-2}F\rrbracket_h\cdot D\varphi\,\dx\dt +
	\tfrac1h\int_\Omega \power{u}{p{-}1}(0)\cdot \int_0^T \mathrm e^{-\frac sh}\varphi\,\ds\,\dx,
\end{align}
for any $\varphi\in L^p(0,T;W^{1,p}_0(\Omega,\R^N))$. 
Let $\eta\in C^1_0(B_R(x_o),[0,1])$ be a cut-off function with $\eta \equiv 1$ in $B_r(x_o)$ and $|D\eta|\leq \frac{2}{R-r}$ and 
$\zeta \in W^{1,\infty} (\Lambda_S(t_o),[0,1])$ defined by 
\begin{equation*}
	\zeta(t)
	:=
	\left\{
	\begin{array}{cl}
	{\displaystyle \frac{t-t_o+S}{S-s} }&
	\mbox{for $t\in (t_o-S,t_o-s)$,}\\[10pt]
	1 & \mbox{for $t\in [t_o- s, t_o+ S)$.}	
	\end{array}
	\right.
\end{equation*}
Furthermore, for $\varepsilon >0$ small enough and $t_1 \in \Lambda_s(t_o)$ we define the function $\psi_\varepsilon \in W^{1,\infty} (\Lambda_S(t_o),[0,1])$ by 
$$
	\psi_\varepsilon(t) 
	:= 
	\left\{
	\begin{array}{cl}
	1 & \mbox{for $t\in (t_o-S, t_1]$,} \\[3pt]
	1-\frac{1}{\varepsilon} (t-t_1) & 
	\mbox{for $t\in (t_1, t_1+\varepsilon)$,} \\[3pt]
	0  & \mbox{for $t\in [t_1+\varepsilon, t_o+S)$.}
\end{array}
\right.
$$
In \eqref{mol-eq} we choose the testing function
$$
	\varphi(x,t) 
	= 
	\eta^p(x) \zeta(t)\psi_\varepsilon(t) \big(u(x,t)-a\big). 
$$
In the following we abbreviate $\power{w}{p{-}1}:=\llbracket \power{u}{p{-}1}\rrbracket_h$
and omit in the notation the reference to the center $z_o=(x_o,t_o)$. For the integral in \eqref{mol-eq} containing the time derivative we compute 
\begin{align*}
	\iint_{\Omega_T} 
	\partial_t \llbracket \power{u}{p{-}1}\rrbracket_h\cdot\varphi \,\dx\dt 
	&= 
	\iint_{Q_{R,S}}
	\eta^p \zeta \psi_\varepsilon \partial_t \power{w}{p{-}1}\cdot (w-a) 
	\dx\dt \\
	&\phantom{=\,} +
	\iint_{Q_{R,S}}
	\eta^p \zeta \psi_\varepsilon \partial_t \power{w}{p{-}1}\cdot (u-w) 
	\dx\dt \\
	&\ge 
	- \iint_{Q_{R,S}}
	\eta^p \zeta \psi_\varepsilon 
	\partial_t\Big(\tfrac{p-1}{p}|w|^{p} - \power{w}{p{-}1}\cdot a\Big) \dx\dt \\
	& = 
	- \iint_{Q_{R,S}}
	\eta^p \zeta \psi_\varepsilon 
	\partial_t \mathfrak b[w,a] \dx\dt \\
	& = 
	\iint_{Q_{R,S}} 
	\eta^p \big( \zeta \psi_\varepsilon' +\psi_\varepsilon\zeta' \big) 
	\mathfrak b[w,a] \dx\dt,
\end{align*}
where we used the identity $\partial_t \power{w}{p{-}1} = -\frac1h(\power{w}{p{-}1}-\power{u}{p{-}1})$ and recall the definition of the boundary term $\mathfrak b$ in \eqref{def:b}.
Since $\llbracket \power{u}{p{-}1}\rrbracket_h\to \power{u}{p{-}1}$ in $L^{\frac{p}{p-1}}(\Omega_T)$ we can pass to the limit $h\downarrow 0$ in the integral on the right-hand side. We therefore get
\begin{align*}
	\liminf_{h\downarrow 0}
	\iint_{Q_{R,S}} 
	\partial_t \llbracket \power{u}{p{-}1}\rrbracket_h\cdot\varphi \,\dx\dt 
	\ge
	\iint_{Q_{R,S}}
	\eta^p\big( \zeta \psi_\varepsilon' +\psi_\varepsilon\zeta' \big) 
	\mathfrak b[u, a] \,\dx\dt 
	=: 
	\mathrm{I}_{\varepsilon} +\mathrm{II}_{\varepsilon}.
\end{align*}
We now pass to the limit $\varepsilon \downarrow 0$.
 For the term $\mathrm{I}_{\varepsilon}$ we obtain for any $t_1 \in \Lambda_{s}$
 that
$$
 	\lim_{\varepsilon \downarrow 0}\mathrm{I}_\varepsilon 
	=
	\int_{B_R} \eta^p\, \mathfrak b[u(t_1), a] \d x.
$$
Taking into account that the boundary term $\mathfrak b[u,a]$ is non-negative, the term $\mathrm{II}_\varepsilon$ can be estimated independently from $\epsilon$, since
\begin{align*}
	|\mathrm{II}_\varepsilon| 
	\leq 
	\iint_{Q_{R,S}} \zeta'\, \mathfrak b[u,a]\, \dx\dt 
	\leq 
	\iint_{Q_{R,S}} 
	\frac{\mathfrak b[u,a]}{S-s}\, \dx\dt.
\end{align*}
Next, we consider the diffusion term. 
After passing to the limit $h\downarrow 0$, we use the ellipticity and growth assumption \eqref{growth-a} for the vector-field $\mathbf A$, and subsequently Young's inequality. In this way, we obtain 
\begin{align*}
	\lim_{h\downarrow 0}
	\iint_{\Omega_T} &
	\llbracket\mathbf A(x,t,u,Du)\rrbracket_h\cdot D\varphi\, \dx\dt \\
	&=
	\iint_{Q_{R,S}} 
	\zeta \psi_\varepsilon \mathbf A (x,t,u,Du) \cdot  
	\big[\eta^p Du + p\eta^{p-1} (u-a)\otimes D\eta\big] 
	\dx\dt \\
	&\geq  
	\nu \iint_{Q_{R,S}}
	\eta^p \zeta \psi_\varepsilon|Du|^p \dx\dt -
	Lp\iint_{Q_{R,S}} 
	\eta^{p-1} \zeta \psi_\varepsilon |D\eta| |u-a| |Du|^{p-1}
	\dx\dt \\
	&\geq   
	\tfrac \nu2 \iint_{Q_{R,S}}
	\eta^p \zeta \psi_\varepsilon |Du|^p \dx\dt -
	c \iint_{Q_{R,S}}\frac{|u-a|^p}{(R-r)^p} \dx\dt.
\end{align*}
Next, we consider the right-hand side term involving the inhomogeneity $F$. With the help of Young's inequality we find that
\begin{align*}
   	\lim_{h\downarrow0}&\bigg|\iint_{\Omega_T} \llbracket
   	|F|^{p-2}F\rrbracket_h\cdot D\varphi\,\dx\dt\bigg| \\
   	&=
   	\iint_{Q_{R,S}} 
   	\zeta \psi_\varepsilon |F|^{p-2} F \cdot
   	\big[\eta^p Du + p\eta^{p-1} (u-a)\otimes D\eta\big]
   	\dx\dt\\
   	&\le
   	\tfrac{\nu}{4}\iint_{Q_{R,S}} \eta^p\zeta \psi_\varepsilon |Du|^p\dx\dt +
    c\,\iint_{Q_{R,S}}\bigg[\frac{|u-a|^p}{(R-r)^p} + |F|^p\bigg] \dx\dt.
\end{align*}
Finally, for the last integral in \eqref{mol-eq}, the convergence of the mollifications and the fact $\varphi(0)=0$ imply 
\begin{equation*}
   	\lim_{h\downarrow0}
   	\tfrac1h\int_\Omega \power{u}{p{-}1}(0)\cdot \int_0^T \mathrm
   	e^{-\frac sh}\varphi\,\ds\,\dx
   	=
   	\int_\Omega\power{u}{p{-}1}(0)\cdot\varphi(0)\,\dx
   	=0.
\end{equation*}
Combining the preceding results and passing to the limit $\epsilon\downarrow 0$ we obtain for almost every $t_1 \in \Lambda_s$ that
\begin{equation*}
	\int_{B_r} \mathfrak b [u(t_1),a] \dx +  
	\int_{t_o-s}^{t_1}\int_{B_r}  
	|Du|^p \dx\dt 
	\leq 
	c \iint_{Q_{R,S}}
	\bigg[\frac{|u-a|^p}{(R-r)^p} +
	\frac{\mathfrak b[u,a]}{S-s}
	+ |F|^p \bigg]
	\dx\dt,
\end{equation*}
for a constant $c=c(p,\nu,L)$. 
Here we pass to the supremum over $t_1\in \Lambda_s$ in the first term on the left-hand side. In the second one we let  $t_1\uparrow t_o+s$. Finally we take mean values on both sides and apply Lemma~\ref{lem:basic-est-b} twice.  This leads to  the claimed energy estimate.
\end{proof}

Next, we deduce a gluing lemma for the doubly nonlinear system.

\begin{lemma}\label{lem:time-diff}
Let $p>1$ and $u$ be a weak solution to \eqref{eq-doubly} in $\Omega_T$ in the sense of Definition~{\upshape\ref{def:weak_solution}}. 
Then, on any cylinder $Q_{R,S}(z_o):=B_R(x_o)\times\Lambda_S(t_o)\subseteq\Omega_T$ with $R,S>0$ 
there exists $\hat r\in [\frac{R}{2},R)$ such that for all $t_1,t_2\in\Lambda_S(t_o)$, we have
\begin{align*}
	\big|\langle\power{u}{p{-}1}\rangle_{x_o;\hat r}(t_2) - 
	\langle\power{u}{p{-}1}\rangle_{x_o;\hat r}(t_1)\big| 
	&\le
	\frac{c\,S}{R}\,
	\biint_{Q_{R,S}(z_o)} \big[|Du|^{p-1} + |F|^{p-1}\big]\dx\dt ,
\end{align*}
where $c=c(L)$.
\end{lemma}

\begin{proof}
Let $t_1,t_2\in\Lambda_S(t_o)$ with $t_1<t_2$ and assume that $r \in [\frac{R}{2}, R)$. For  $0<\varepsilon,\delta \ll 1$, we define $\xi_\epsilon\in W^{1,\infty}(\Lambda_S(t_o))$ by 
$$
	\xi_\epsilon(t)
	:=
	\left\{
	\begin{array}{cl}
	0 ,& \mbox{for $t_o-S\le t\le t_1-\epsilon$,}\\[3pt]
	\frac{t-t_1+\epsilon}{\epsilon}, & \mbox{for $t_1-\epsilon< t< t_1$,}\\[3pt]
	1 ,& \mbox{for $t_1\le t\le t_2$,}\\[3pt]
	\frac{t_2+\epsilon-t}{\epsilon}, & \mbox{for $t_2< t< t_2+\epsilon$,}\\[3pt]
	0 ,& \mbox{for $t_2+\epsilon\le t\le t_o+S$}
	\end{array}
	\right.
$$
and a radial function $\Psi_\delta \in W^{1,\infty}_0(B_{r+\delta}(x_o))$ by $\Psi_\delta(x):=\psi_\delta(|x-x_o|)$, where
$$
	\psi_\delta(s)
	:=
	\left\{
	\begin{array}{cl}
	1 ,& \mbox{for $0\le s\le r$,}\\[3pt]
	\frac{r+\delta-s}{\delta}, & \mbox{for $r< s< r+\delta$,}\\[3pt]
	0 ,& \mbox{for $r+\delta\le s\le R$.}
	\end{array}
	\right.
$$
For fixed $i\in\{1,\dots,N\}$ we choose $\varphi_{\epsilon,\delta}=\xi_\epsilon\psi_\delta e_i$ as testing function in the weak formulation \eqref{weak-solution}, where $e_i$ denotes the $i$-th canonical basis
vector in $\R^N$. In the limit $\varepsilon,\delta \downarrow 0$ we obtain
\begin{align*}
  \int_{B_{r}(x_o)} &\big[\power{u}{p{-}1}(t_2)-\power{u}{p{-}1}(t_1)\big]\cdot e_i \,\d x \\
  &= 
  \int_{t_1}^{t_2}\int_{\partial B_{r}(x_o)}
  \big[\mathbf A(x,t,u,Du) + |F|^{p-2}F\big]\cdot e_i\otimes\frac{x-x_o}{|x-x_o|}
  \,\d \mathcal{H}^{n-1}(x)\d t ,
\end{align*}
for a.e. $r\in[\frac R2,R)$.  
Multiplying the preceding inequality by $e_i$ and summing over $i=1,\dots ,N$ yields
\begin{align*}
	\int_{B_{r}(x_o)} &\big[\power{u}{p{-}1}(t_2)-\power{u}{p{-}1}(t_1)\big] \d x \\
	&= 
	\int_{t_1}^{t_2}\int_{\partial B_{r}(x_o)}
	\big[\mathbf A(x,t,u,Du) + |F|^{p-2}F\big] \frac{x-x_o}{|x-x_o|} 
	\,\d \mathcal{H}^{n-1}(x)\d t.
\end{align*} 
Due to growth condition \eqref{growth-a}$_2$ we get  
for any $t_1,t_2\in\Lambda_S(t_o)$ and a.e. $r \in [\frac{R}{2}, R)$ that 
\begin{align*}
	\bigg|\int_{B_{r}(x_o)} &\big[\power{u}{p{-}1}(t_2)-\power{u}{p{-}1}(t_1)\big] \d x\bigg| 
	\le
	\int_{\Lambda_S(t_o)}\int_{\partial B_{r}(x_o)}
	\big[L|Du|^{p-1} + |F|^{p-1}\big]\,\d \mathcal{H}^{n-1}\d t .
\end{align*} 
Since 
\begin{align*}
	\int_{\Lambda_S(t_o)} & \int_{B_R(x_o)}
	\big[L|Du|^{p-1} + |F|^{p-1}\big] \dx\dt	\\
	&=
	\int_0^R\int_{\Lambda_S(t_o)}\int_{\partial B_{r}(x_o)} 
	\big[L|Du|^{p-1} + |F|^{p-1}\big] \d \mathcal{H}^{n-1}\d t\d r \\
	&\ge
	\int_{R/2}^R\int_{\Lambda_S(t_o)}\int_{\partial B_{r}(x_o)} 
	\big[L|Du|^{p-1} + |F|^{p-1}\big] \d \mathcal{H}^{n-1}\d t\d r ,
\end{align*}
there exists  a radius $\hat r \in [\frac{R}{2},R)$ with
\begin{align*}
	\int_{\Lambda_S(t_o)} & \int_{\partial B_{\hat r}(x_o)} 
	\big[L|Du| ^{p-1}+ |F|^{p-1}\big] \d \mathcal{H}^{n-1}\d t \\
	&\leq 
	\tfrac{2}{R} \int_{\Lambda_S(t_o)}\int_{B_R(x_o)}
	\big[L|Du| ^{p-1}+ |F|^{p-1}\big] \dx\dt.
\end{align*}
In the above inequality, we choose $r=\hat r$  and then take means on
both sides of the resulting estimate.
This implies 
\begin{align*}
	\big|\langle\power{u}{p{-}1}\rangle_{x_o;\hat r}(t_2)-
	\langle\power{u}{p{-}1}\rangle_{x_o;\hat r}(t_1)\big| 
	&\leq  
	\frac{c}{R} \int_{\Lambda_S(t_o)}\bint_{B_R(x_o)} 
	\big[|Du|^{p-1} + |F|^{p-1}\big] \dx\dt \\
	&=
	\frac{c\,S}R
	\biint_{Q_{R,S}(z_o)} 
	\big[|Du|^{p-1} + |F|^{p-1}\big] \dx\dt,
\end{align*}
for any $t_1,t_2\in\Lambda_S(t_o)$ and with a constant $c=c(L)$. 
\end{proof}

\section{Parabolic Sobolev-Poincar\'e type inequalities}\label{sec:poin}

One of the difficulties in the parabolic setting is that weak solutions are not necessarily differentiable with respect to time. As a consequence, the Sobolev-Poincar\'e inequality on $\R^{n+1}$ is not applicable. Since such an inequality is indispensable in the proof of the higher integrability we will derive some type of Poincar\'e and Sobolev-Poincar\'e inequality which is valid for weak solutions. The idea is to use the Gluing Lemma~\ref{lem:time-diff} in order to manage the lack of differentiability with respect to time. 

Throughout this section we consider scaled cylinders $Q_{\rho}^{(\mu)}(z_o)\subseteq\Omega_T$ as defined in \eqref{cylinder} on which  certain intrinsic, respectively sub-intrinsic couplings with respect to $u$ and its spatial gradient $Du$ hold true.  
For $\rho,\mu>0$ we assume that
\begin{equation}\label{subintrinsic-mu-poin}
	\frac{\displaystyle{\biint_{Q_{\rho}^{(\mu)}(z_o)} 
	\frac{\abs{u}^{p}}{\mu^{2-\pf}\rho^p}\dx\dt}}{\displaystyle{\biint_{Q_{\rho}^{(\mu)}(z_o)} \big[ |Du|^p +|F|^p\big]\dx\dt}}
	\le 
	K\mu^{p}
\end{equation}
holds true for a constant $K\ge 1$. Recall that $\pf=\max\{2,p\}$. Such cylinders are termed $\mu$-sub-intrinsic. Furthermore, we assume that either 
\begin{equation}\label{superintrinsic-mu-poin}
	\mu^{p}
	\le
	K
	\frac{\displaystyle{\biint_{Q_{\rho}^{(\mu)}(z_o)} 
	\frac{\abs{u}^{p}}{\mu^{2-\pf}\rho^p}\dx\dt}}{\displaystyle{\biint_{Q_{\rho}^{(\mu)}(z_o)} \big[ |Du|^p +|F|^p\big]\dx\dt}}
	\qquad\mbox{or}\qquad
	\mu^{p}
	\le
	K 
\end{equation}
holds true. A cylinder $Q_{\rho}^{(\mu)}(z_o)$ satisfying
\eqref{superintrinsic-mu-poin}$_1$ is called $\mu$-super-intrinsic. Finally, a cylinder which is $\mu$-sub- and $\mu$-super-intrinsic is called $\mu$-intrinsic. In
the following we distinguish the cases whether the growth exponent $p$
is sub- or superquadratic.
In order to emphasize the stability of the proof when $p\to2$, we include the quadratic case $p=2$ in both subsections. 

\subsection{The case $\pmb{\max\{\frac{2n}{n+2},1\}<p\le 2}$}

As a first preliminary result, we compare the first and the second
term on the right-hand side of the energy inequality in
Lemma~\ref{lem:energy}. It turns out that for $p\in(1,2]$ on
$\mu$-sub-intrinsic cylinders the second term can easily be bounded in terms of the first one.

\begin{lemma}\label{lem:2<1-1}
Let $1<p\le2$ and $u$ be a weak solution to \eqref{eq-doubly} in $\Omega_T$ in the sense of Definition~{\upshape\ref{def:weak_solution}}. 
Then, on any cylinder $Q_{\rho}^{(\mu)}(z_o)\subseteq\Omega_T$ with $\rho,\mu>0$ satisfying
the $\mu$-sub-intrinsic coupling \eqref{subintrinsic-mu-poin}, we have
\begin{align*}
    &\biint_{Q_\rho^{(\mu)}(z_o)} 
	\frac{\big|u-(u)_{z_o;\rho}^{(\mu)}\big|^{p}}{\rho^p}\, \dx\dt\\
	&\quad\le
	c \Bigg[
    \biint_{Q_\rho^{(\mu)}(z_o)}
    \frac{\big|\power{u}{\frac{p}{2}}-
    \power{\big[(u)_{z_o;\rho}^{(\mu)}\big]}{\frac p2}\big|^2}
    {\mu^{p-2}\rho^p} \dx\dt
    \Bigg]^{\frac p2}
    \bigg[\biint_{Q_{\rho}^{(\mu)}(z_o)} \big[ |Du|^p +|F|^p\big]\dx\dt
	\bigg]^\frac{2-p}{2}
\end{align*}
where $c= c(p,K)$. 
\end{lemma}

\begin{proof}
For simplicity in notation, we omit the reference point $z_o$. Due to Lemma~\ref{lem:Acerbi-Fusco} applied with $\alpha=\frac2p$ and H\"older's inequality, we obtain
\begin{align*}
	&\biint_{Q_\rho^{(\mu)}} 
	\frac{\big|u-(u)_{\rho}^{(\mu)}\big|^{p}}{\rho^p}\, \dx\dt \\
	&\quad\le 
	c\,\biint_{Q_\rho^{(\mu)}} 
	\frac{\big|\power{u}{\frac p2}-
	\power{\big[(u)_{\rho}^{(\mu)}\big]}{\frac p2}\big|^{p}
	\big[|u|+|(u)_{\rho}^{(\mu)}|\big]^{\frac{p(2-p)}{2}}}
	{\rho^p}\, \dx\dt \\
	&\quad\le 
	c\,\Bigg[
	\biint_{Q_\rho^{(\mu)}} 
	\frac{\big|\power{u}{\frac p2}-
	\power{\big[(u)_{\rho}^{(\mu)}\big]}{\frac p2}\big|^{2}}
	{\rho^p}\, \dx\dt \Bigg]^{\frac p2} 
	\Bigg[
	\biint_{Q_\rho^{(\mu)}} 
	\frac{\big[|u|+|(u)_{\rho}^{(\mu)}|\big]^{p}}
	{\rho^p}\, \dx\dt \Bigg]^\frac{2-p}{2} .
\end{align*}
For the last integral, hypothesis~\eqref{subintrinsic-mu-poin} yields
\begin{align*}
	\biint_{Q_\rho^{(\mu)}} 
	\frac{\big[|u|+|(u)_{\rho}^{(\mu)}|\big]^{p}}
	{\rho^p}\, \dx\dt 
	&\le 
	2^p
	\biint_{Q_\rho^{(\mu)}} 
	\frac{|u|^{p}}
	{\rho^p}\, \dx\dt  \\
	&\le 
	2^p K\mu^p
	\biint_{Q_{\rho}^{(\mu)}} \big[ |Du|^p +|F|^p\big]\dx\dt .
\end{align*}
Inserting this above proves the claim with a constant $c$ depending only on $p$ and $K$.
\end{proof}

The next lemma should be interpreted as a parabolic Poincar\'e inequality for solutions on $\mu$-sub-intrinsic cylinders. The fact that weak solutions do not necessarily possess a weak time derivative is compensated by the Gluing Lemma~\ref{lem:time-diff}. However, the gluing lemma provides an estimate for time differences of slice-wise means of $\power{u}{p{-}1}$ rather than $u$. Therefore, mean values of $\power{u}{p{-}1}$ and $u$ have to be estimated very carefully against each other.

\begin{lemma}\label{lem:poin-1}
Let $1<p\le2$ and $u$ be a weak solution to \eqref{eq-doubly} in $\Omega_T$ in the sense of Definition~{\upshape\ref{def:weak_solution}}. 
Then, on any cylinder $Q_{\rho}^{(\mu)}(z_o)\subseteq\Omega_T$ satisfying
the $\mu$-sub-intrinsic coupling \eqref{subintrinsic-mu-poin}, the inequality  
\begin{align*}
	&\biint_{Q_\rho^{(\mu)}(z_o)} 
	\frac{\big|u-(u)_{z_o;\rho}^{(\mu)}\big|^{q}}{\rho^q}\, \dx\dt \\
	&\quad\le
	c\,\bigg[\biint_{Q_{\rho}^{(\mu)}(z_o)} 
	\big[|Du|^{q} + |F|^{q}\big]\dx\dt
	\bigg]^{p-1}
	\bigg[\biint_{Q_{\rho}^{(\mu)}(z_o)} 
	\big[ |Du|^p +|F|^p\big]\dx\dt
	\bigg]^{\frac{q(2-p)}{p}}  ,
\end{align*}
holds true for any $q\in[1,p]$ and a constant $c=c(n,p,L,K)$. 
\end{lemma}

\begin{proof}
In the proof we renounce again to consider the center $z_o$ in the notation.
With $\hat\rho\in [\frac\rho2, \rho)$ we denote the radius from Lemma \ref{lem:time-diff}.
We start by estimating the left-hand side with the help of the
quasi-minimality of the mean value as follows
\begin{align*}
     \biint_{Q_\rho^{(\mu)}}&
     \frac{\big|u-(u)_{\rho}^{(\mu)}\big|^q}{\rho^q}\,\d x\d t
     \le
     c\,\biint_{Q_\rho^{(\mu)}}
     \frac{\big|u-
     \power{\big[(u^{p-1})_{\hat\rho}^{(\mu)}\big]}{\frac{1}{p-1}}
     \big|^q}
     {\rho^{q}}\,\d x\d t 
     \le 
     c\,[\mathrm{I}+\mathrm{II}],
\end{align*}
where we abbreviated
\begin{align*}
	\mathrm{I}
    &:=
    \biint_{Q_\rho^{(\mu)}}
	\frac{\big|u-\power{\langle u^{p-1}\rangle_{\hat\rho}}{\frac{1}{p-1}}(t)\big|^q}
    {\rho^{q}}\,\d x\d t,\nonumber\\
    \mathrm{II}
    &:=
    \bint_{\Lambda_\rho^{(\mu)}}
    \frac{\big|\power{\langle u^{p-1}\rangle_{\hat\rho}}{\frac{1}{p-1}}(t)-
    \power{\big[(u^{p-1})_{\hat\rho}^{(\mu)}\big]}{\frac{1}{p-1}}\big|^q}
    {\rho^{q}}\,\d t.
\end{align*}
Next, we treat the terms $\mathrm{I}$ and $\mathrm{II}$ of the right-hand side.
For the term $\mathrm{I}$ we first recall that $\hat\rho\in [\frac{\rho}{2},\rho]$. Therefore, the application of Lemma \ref{lem:alphalemma} with $\alpha= \frac{1}{p-1}\ge \frac1q$ and subsequently Poincar\'e's inequality leads to
\begin{align*}
	\mathrm{I}
	\le
	c\,\biint_{Q_\rho^{(\mu)}}
	\frac{|u-\langle u\rangle_{\rho}(t)|^q}{\rho^{q}}\,\dx\dt 
	\le
	c\,\biint_{Q_\rho^{(\mu)}}
	|Du|^q\,\dx\dt  ,
\end{align*}
for a constant $c=c(n,p)$. Note that $q\in[1,p]$ and the constant in Poincar\'e's inequality depends continuously on $q$. Now we will treat $\mathrm{II}$.
An application of Lemma~\ref{lem:Acerbi-Fusco} with $\alpha=\frac{1}{p-1}\ge 1$ and subsequently H\"older's inequality yields
\begin{align*}
    \mathrm{II}
    &\le
    \frac{c}{\rho^q} \bint_{\Lambda_\rho^{(\mu)}}
    \Big[\big|\langle\power{u}{p{-}1}\rangle_{\hat\rho}(t)\big| + 
    \big|(\power{u}{p{-}1})_{\hat\rho}^{(\mu)}\big|\Big]^{\frac{q(2-p)}{p-1}}
    \big|\langle\power{u}{p{-}1}\rangle_{\hat\rho}(t)-
    (\power{u}{p{-}1})_{\hat\rho}^{(\mu)}\big|^q\d t \\
    &\le
    \frac{c}{\rho^q} 
    \sup_{t,\tau\in \Lambda_\rho^{(\mu)}} \!
    \big|\langle\power{u}{p{-}1}\rangle_{\hat\rho}(t)-
    \langle\power{u}{p{-}1}\rangle_{\hat\rho}(\tau)\big|^q
    \bint_{\Lambda_\rho^{(\mu)}}\!\!
    \Big[\big|\langle\power{u}{p{-}1}\rangle_{\hat\rho}(t)\big| + 
    \big|(\power{u}{p{-}1})_{\hat\rho}^{(\mu)}\big|\Big]^{\frac{q(2-p)}{p-1}} \dt \\
    &\le
    \frac{c}{\rho^{q(p-1)}} 
    \sup_{t,\tau\in \Lambda_\rho^{(\mu)}} 
    \big|\langle\power{u}{p{-}1}\rangle_{\hat\rho}(t)-
    \langle\power{u}{p{-}1}\rangle_{\hat\rho}(\tau)\big|^q
    \bigg[\biint_{Q_\rho^{(\mu)}} \frac{|u|^p}{\rho^p} \dx\dt
    \bigg]^{\frac{q(2-p)}{p}} ,
\end{align*}
for a constant $c(p)$. We continue estimating the right-hand side with the help of the Gluing Lemma~\ref{lem:time-diff}, the $\mu$-sub-intrinsic coupling \eqref{subintrinsic-mu-poin} and H\"older's inequality. In this way we find 
\begin{align*}
    \mathrm{II}
    &\le
    c\,\mu^{q(p-2)}
    \bigg[
	\biint_{Q_{\rho}^{(\mu)}} \big[|Du|^{p-1} + |F|^{p-1}\big]\dx\dt
	\bigg]^q
    \bigg[\biint_{Q_\rho^{(\mu)}} \frac{|u|^p}{\rho^p} \dx\dt
    \bigg]^{\frac{q(2-p)}{p}} \\
    &\le
    c\, \bigg[\biint_{Q_{\rho}^{(\mu)}} 
    \big[|Du|^{q} + |F|^{q}\big]\dx\dt\bigg]^{p-1} 
    \bigg[\biint_{Q_{\rho}^{(\mu)}} \big[ |Du|^p +|F|^p\big]\dx\dt
	\bigg]^{\frac{q(2-p)}{p}} ,
\end{align*}
for a constant $c(p,L,K)$.       
Joining the preceding estimates for $\mathrm{I}$ and $\mathrm{II}$ finally proves the claim.
\end{proof}

Our next aim is to derive a Sobolev-Poincar\'e type inequality. It has to be understood in the following way. Lemma~\ref{lem:2<1-1} allows to bound the second term on the right-hand side of the energy inequality in terms of the first one. Therefore, in our Sobolev-Poincar\'e type inequality we will derive an upper bound for this term. In this bound we would like to have the integral of $|Du|^q$ for some $q<p$ on the right-hand side. However, due to the nonhomogeneous behavior of the underlying differential equation some extra terms show up. Fortunately they have exactly the form of the left-hand side of the energy estimate so that they can be re-absorbed later on. Note that the estimate of the term $\mathrm{II}_2$ in the proof of Lemma~\ref{lem:sob-1} is the only point in the paper where the condition $p>\frac{2n}{n+2}$ is needed.

\begin{lemma}\label{lem:sob-1}
Let $\max\{\frac{2n}{n+2},1\}<p\le2$ and $u$ be a weak solution to \eqref{eq-doubly} in $\Omega_T$ in the sense of Definition~{\upshape\ref{def:weak_solution}}. 
Then, on any cylinder $Q_{\rho}^{(\mu)}(z_o)\subseteq\Omega_T$ with $\rho, \mu>0$ satisfying
\eqref{subintrinsic-mu-poin} and \eqref{superintrinsic-mu-poin} and for any $\epsilon\in(0,1]$, we have
\begin{align*}
    &
    \biint_{Q_\rho^{(\mu)}(z_o)}
    \frac{\big|\power{u}{\frac{p}{2}}-
    \power{\big[(u)_{z_o;\rho}^{(\mu)}\big]}{\frac p2}\big|^2}
    {\mu^{p-2}\rho^p}\,\d x\d t\\
	&\quad\le
	\epsilon\Bigg[\sup_{t\in \Lambda_\rho^{(\mu)}(t_o)}
	\bint_{B_\rho(x_o)}
	\frac{\big|\power{u}{\frac{p}{2}}(t)-
	\power{\big[(u)_{z_o;\rho}^{(\mu)}\big]}{\frac{p}{2}}\big|^2}
	{\mu^{p-2}\rho^p} \dx +
	\biint_{Q_\rho^{(\mu)}(z_o)} |Du|^{p} \dx\dt\Bigg] \\
	&\quad\quad+
	\frac{c}{\epsilon^{\frac{2-q(p-1)}{q(p-1)}}} 
	\Bigg[
	\bigg[\biint_{Q_\rho^{(\mu)}(z_o)} |Du|^{q} \dx \dt 
	\bigg]^{\frac{p}{q}} +
	\biint_{Q_\rho^{(\mu)}(z_o)} |F|^p \dx\dt \Bigg]
\end{align*}
with $q=\max\{\frac{2n}{n+2},1\}$ and $c=c(n,p,L, K)$. 
\end{lemma}

\begin{proof}
As before, we omit the reference point $z_o$ in our notation. Moreover, we abbreviate $a:=(u)_{\rho}^{(\mu)}$ and define
\begin{align*}
	\mathrm{I}
	&
	:=
	 \sup_{t\in \Lambda_\rho^{(\mu)}}
          \bint_{B_\rho}
          \frac{|\power{u}{\frac{p}{2}}(t)-\power{a}{\frac p2}|^2}
         {\mu^{p-2}\rho^p}\,\d x,
\end{align*}
and
\begin{align*}
	\mathrm{II}
	&
	:=
	\Bigg[\bint_{\Lambda_\rho^{(\mu)}}
    \bigg[ \bint_{B_\rho}
    \frac{|\power{u}{\frac{p}{2}}-\power{a}{\frac p2}|^2}{\rho^p}\,\dx
    \bigg]^{\frac{q}{2}} \dt \Bigg]^{\frac{2}{q}}.
\end{align*}
Then, we obtain
\begin{align*}
    	\biint_{Q_\rho^{(\mu)}}
    	\frac{|\power{u}{\frac{p}{2}}-\power{a}{\frac p2}|^2}
    	{\mu^{p-2}\rho^p}\,\d x\d t
    	\le
    	\mathrm{I}^\frac{2-q}{2}\cdot 
	\big[\mu^{2-p}\,\mathrm{II}\big]^\frac{q}{2}.
\end{align*}
In the following, it remains to consider the second term on the right-hand side. 
For the estimate of $\mu$ we use hypothesis \eqref{superintrinsic-mu-poin}. If \eqref{superintrinsic-mu-poin}$_1$ is satisfied, we first apply Lemma~\ref{lem:poin-1} with $q=p$ to obtain
\begin{align*}
	\bigg[\biint_{Q_{\rho}^{(\mu)}} \frac{\abs{u}^{p}}{\rho^p}\dx\dt
	\bigg]^\frac1{p} 
	&\le
	\bigg[ 
	\biint_{Q_{\rho}^{(\mu)}} 
	\frac{|u - a|^{p}}{\rho^p}\dx\dt\bigg]^\frac1{p}
	+
	\frac{|a|}{\rho} \\
	&\le
	c\,\bigg[ 
    \biint_{Q_{\rho}^{(\mu)}} \big[|Du|^p + |F|^p\big] \dx\dt
    \bigg]^\frac1{p}
	+
	\frac{|a|}{\rho} .
\end{align*}
Together with the $\mu$-super-intrinsic coupling \eqref{superintrinsic-mu-poin}$_1$ this yields
\begin{align}\label{mu-est}
	\mu 
	\le
	c
	+
	\frac{c\,|a|}{
	\rho
	\displaystyle{\bigg[\biint_{Q_{\rho}^{(\mu)}} \big[|Du|^p + |F|^p\big] \dx\dt
    \bigg]^{\frac1{p}}}}  \,,
\end{align}
with a constant $c$ depending on $n,p,L$ and $K$. On the other
hand, if \eqref{superintrinsic-mu-poin}$_2$ is satisfied, then \eqref{mu-est} holds true with $c=K^{1/p}$. Consequently, we have inequality \eqref{mu-est} in any case and therefore obtain 
\begin{align*}
    \mu^{2-p}\,\mathrm{II} 
    \le
    c\, \big[\mathrm{II}_1 + \mathrm{II}_2\big],
\end{align*}
where 
\begin{align*}
    \mathrm{II}_1
    &:=
	\Bigg[\bint_{\Lambda_\rho^{(\mu)}}
    \bigg[ \bint_{B_\rho}
    \frac{|\power{u}{\frac{p}{2}}-\power{a}{\frac p2}|^2}
    {\rho^p}\,\dx\bigg]^{\frac{q}{2}} \dt\Bigg]^{\frac{2}{q}} 
\end{align*}
and
\begin{align*}
    \mathrm{II}_2
    &:=
	\frac{
    \Bigg[\displaystyle{\bint_{\Lambda_\rho^{(\mu)}}
    \bigg[ 
    \bint_{B_\rho}
    \frac{|a|^{2-p}|\power{u}{\frac{p}{2}}-\power{a}{\frac p2}|^2}
    {\rho^2}\,\dx\bigg]^{\frac{q}{2}} \dt}\Bigg]^{\frac{2}{q}}}{
    \bigg[\displaystyle{\biint_{Q_{\rho}^{(\mu)}} \big[|Du|^p + |F|^p\big] \dx\dt}
    \bigg]^{\frac{2-p}{p}}} .
\end{align*}
For the estimate of $\mathrm{II}_1$ we apply Lemma~\ref{lem:a-b} with
$\alpha =\frac2p$, Sobolev's inequality and Lemma~\ref{lem:poin-1}. In this way we find 
\begin{align*}
	\mathrm{II}_1
    &\le 
    c\, \Bigg[\bint_{\Lambda_\rho^{(\mu)}} 
    \bigg[ \bint_{B_\rho}
    \frac{|u-a|^p}{\rho^p}\,\dx
    \bigg]^{\frac{q}{2}} \dt 
    \Bigg]^{\frac{2}{q}}\nonumber\\
    &\le 
    c\, \Bigg[\bint_{\Lambda_\rho^{(\mu)}}
    \bigg[ \bint_{B_\rho}
    \bigg[|Du|^q +
    \frac{|u-a|^q}{\rho^q}
    \bigg]\,\dx\bigg]^{\frac{p}{2}} \dt 
    \Bigg]^{\frac{2}{q}}\nonumber\\
    &\le 
    c\,\Bigg[\biint_{Q_\rho^{(\mu)}}
    \bigg[ 
     |Du|^q + \frac{|u-a|^q}{\rho^q}
    \bigg]\,\dx \dt\Bigg]^{\frac{p}{q}} \nonumber\\
    &\le
	c\,
	\bigg[\biint_{Q_\rho^{(\mu)}} 
	\big[|Du|^q + |F|^q\big] \,\dx\dt
	\bigg]^{\frac{p(p-1)}{q}} 
	\bigg[\biint_{Q_\rho^{(\mu)}} 
	\big[|Du|^p + |F|^p\big] \,\dx\dt
	\bigg]^{2-p} ,
\end{align*}
where $c= c(n,p,L,K)$.
Now we turn our attention to the second term. With the help of Lemma~\ref{lem:Acerbi-Fusco} applied with $\alpha=\frac2p$ and Sobolev's inequality, we find that
\begin{align*}
    \mathrm{II}_2
    &\le
    c\,\frac{
    \Bigg[\displaystyle{
    \bint_{\Lambda_\rho^{(\mu)}}
    \bigg[ \bint_{B_\rho}
    \frac{|u-a|^2}
    {\rho^2}\,\dx\bigg]^{\frac{q}{2}} \dt\Bigg]^{\frac{2}{q}}}}{
    \displaystyle{
    \bigg[\biint_{Q_{\rho}^{(\mu)}} \big[|Du|^p + |F|^p\big] \dx\dt
    \bigg]^{\frac{2-p}{p}}}} 
    \le
    c\,\frac{
    \Bigg[\displaystyle{
    \biint_{Q_\rho^{(\mu)}}
    \bigg[ 
    |Du|^{q}+\frac{|u-a|^{q}}{\rho^{q}}
    \bigg] \dx \dt}\Bigg]^{\frac{2}{q}}}{
    \displaystyle{\bigg[\biint_{Q_{\rho}^{(\mu)}} \big[|Du|^p + |F|^p\big] \dx\dt}
    \bigg]^{\frac{2-p}{p}}} 
\end{align*}
for a constant $c=c(n,p)$.
The term involving $|u-a|^q$ is now treated as above with Lemma~\ref{lem:poin-1}, so that 
\begin{align*}
    \mathrm{II}_2
    &\le
    c\,
	\bigg[\biint_{Q_\rho^{(\mu)}} 
	\big[|Du|^q + |F|^q\big] \,\dx\dt\bigg]^{\frac{2(p-1)}{q}} 
	\bigg[\biint_{Q_{\rho}^{(\mu)}} \big[|Du|^p + |F|^p\big] \dx\dt
    \bigg]^{\frac{2-p}{p}}
\end{align*}
holds true with a constant $c$ depending only on $n,p,L$, and $K$.
Inserting the preceding estimates above and applying Young's
inequality, we derive the desired inequality.
\end{proof}

\subsection{The case $\pmb{p\ge 2}$}

Now, we turn our attention to the superquadratic case $p\ge 2$. We emphasize that all results of this section hold true for the full range $p\ge 2$. The restriction $p<\frac{2n}{(n-2)_+}$ will be necessary later on in the covering argument.  
Contrary to the subquadratic case in Lemma~\ref{lem:2<1-1}, we find in the superquadratic case a straight-forward bound on $\mu$-sub-intrinsic  cylinders for the first term on the right-hand side of the energy inequality in Lemma~\ref{lem:energy} in terms of the second one. 

\begin{lemma}\label{lem:1<2-2}
Let $p\ge2$ and $u$ be a weak solution to \eqref{eq-doubly} in $\Omega_T$ in the sense of Definition~{\upshape\ref{def:weak_solution}}. 
Then, on any cylinder $Q_{\rho}^{(\mu)}(z_o)\subseteq\Omega_T$ with $\rho,\mu>0$ satisfying
the $\mu$-sub-intrinsic coupling \eqref{subintrinsic-mu-poin}, 
we kave
\begin{align*}
    &\biint_{Q_\rho^{(\mu)}(z_o)}
    \frac{\big|\power{u}{\frac{p}{2}}-
    \power{\big[(u)_{z_o;\rho}^{(\mu)}\big]}{\frac p2}\big|^2}
    {\rho^p}\,\d x\d t\\
	&\qquad\le
	c\,
	\bigg[
	\biint_{Q_\rho^{(\mu)}(z_o)} 
	\frac{\big|u-(u)_{z_o;\rho}^{(\mu)}\big|^{p}}{\mu^{2-p}\rho^p}\, \dx\dt
	\bigg]^\frac2p 
	\bigg[\biint_{Q_{\rho}^{(\mu)}(z_o)} \big[ |Du|^p +|F|^p\big]\dx\dt
	\bigg]^\frac{p-2}{p}
\end{align*} 
with $c= c(p,K)$. 
\end{lemma}

\begin{proof}
As before, we omit the reference point $z_o$ in our notation. 
Applying Lemma~\ref{lem:Acerbi-Fusco} with $\alpha=\frac p2$,
H\"older's inequality, and finally hypothesis \eqref{subintrinsic-mu-poin}, we obtain
\begin{align*}
	&
    \biint_{Q_\rho^{(\mu)}}
    \frac{\big|\power{u}{\frac{p}{2}}-
	\power{\big[(u)_{\rho}^{(\mu)}\big]}{\frac p2}\big|^2}
    {\rho^p}\,\dx\dt \\
    &\qquad\le
    \frac{c}{\mu^{p-2}}\,\biint_{Q_\rho^{(\mu)}}
    \frac{\big|u-(u)_{\rho}^{(\mu)}\big|^2
    \big[|u|+|(u)_{\rho}^{(\mu)}|\big]^{p-2}}
    {\mu^{2-p}\rho^p}\,\dx\dt \\
    &\qquad\le
    \frac{c}{\mu^{p-2}}\,
    \bigg[\biint_{Q_\rho^{(\mu)}}
    \frac{\big|u-(u)_{\rho}^{(\mu)}\big|^p}{\mu^{2-p}\rho^p}\,\dx\dt
    \bigg]^{\frac{2}{p}} 
    \bigg[\biint_{Q_\rho^{(\mu)}}
    \frac{\big[|u|+|(u)_{\rho}^{(\mu)}|\big]^{p}}{\mu^{2-p}\rho^p}\,\dx\dt
    \bigg]^{\frac{p-2}{p}} \\
    &\qquad\le
    c\,
    \bigg[\biint_{Q_\rho^{(\mu)}}
    \frac{\big|u-(u)_{\rho}^{(\mu)}\big|^p}{\mu^{2-p}\rho^p}\,\dx\dt
    \bigg]^{\frac{2}{p}}
    \bigg[\biint_{Q_{\rho}^{(\mu)}} \big[ |Du|^p +|F|^p\big]\dx\dt
	\bigg]^\frac{p-2}{p} ,
\end{align*}
which proves the claim.
\end{proof}

The next lemma is the analogue of Lemma~\ref{lem:poin-1} for the superquadratic case and should be interpreted as a parabolic Poincar\'e type inequality. 

\begin{lemma}\label{lem:poin-2}
Let $p\ge2$ and $u$ be a weak solution to \eqref{eq-doubly} in $\Omega_T$ in the sense of Definition~{\upshape\ref{def:weak_solution}}. 
Then, on any cylinder $Q_{\rho}^{(\mu)}(z_o)\subseteq\Omega_T$ with $\rho,\mu>0$ satisfying
\eqref{subintrinsic-mu-poin} and \eqref{superintrinsic-mu-poin}, where
\begin{align*}
	\biint_{Q_\rho^{(\mu)}(z_o)} 
	\frac{\big|u-(u)_{z_o;\rho}^{(\mu)}\big|^{q}}{\mu^{\frac{q(2-p)}{p}}\rho^q}\, \dx\dt 
	&\le
	c\,\biint_{Q_{\rho}^{(\mu)}(z_o)} 
	\big[|Du|^{q} + |F|^{q}\big]\dx\dt ,
\end{align*}
for any $q\in[p-1,p]$ and $c=c(n,p,L,K)$. 
\end{lemma}

\begin{proof}
Throughout  the proof we omit the reference to the center $z_o$ in our notation. Similar to the proof of Lemma~\ref{lem:poin-1}, we find that 
\begin{align*}
     \biint_{Q_\rho^{(\mu)}}&
     \frac{\big|u-(u)_{\rho}^{(\mu)}\big|^q}
     {\mu^{\frac{q(2-p)}{p}}\rho^q}\,\d x\d t
     \le
     \biint_{Q_\rho^{(\mu)}}
     \frac{\big|u-
     \power{\big[(u^{p-1})_{\hat\rho}^{(\mu)}\big]}{\frac{1}{p-1}}
     \big|^q}
     {\mu^{\frac{q(2-p)}{p}}\rho^{q}}\,\d x\d t 
     \le 
     2^{q-1}[\mathrm{I}+\mathrm{II}],
\end{align*}
where we abbreviated
\begin{align*}
	\mathrm{I}
    &:=
    \biint_{Q_\rho^{(\mu)}}
	\frac{\big|u-
	\power{\big[\langle u^{p-1}\rangle_{\hat\rho}^{(\mu)}(t)\big]}{\frac{1}{p-1}}
	\big|^q}
    {\mu^{\frac{q(2-p)}{p}}\rho^{q}}\,\d x\d t,\nonumber\\[5pt]
    \mathrm{II}
    &:=
    \mu^{\frac{q(p-2)}{p}}\bint_{\Lambda_\rho}
    \frac{\big|
    \power{\big[\langle u^{p-1}\rangle_{\hat\rho}^{(\mu)}(t)\big]}{\frac{1}{p-1}}-
    \power{\big[(u^{p-1})_{\hat\rho}^{(\mu)}\big]}{\frac{1}{p-1}}\big|^q}
    {\rho^{q}}\,\d t
\end{align*}
with $\hat\rho\in [\frac\rho2, \rho)$ denoting the radius from the Gluing Lemma \ref{lem:time-diff}. 
Recall the abbreviation $\langle u\rangle_{\rho}^{(\mu)}(t):=\langle u\rangle_{B_\rho^{(\mu)}}(t)$ for the slice-wise mean. 
For the estimate of $\mathrm{I}$ we in turn apply Lemma~\ref{lem:alphalemma} with $\alpha= \frac{1}{p-1}\ge \frac1q$ and $\hat\rho\in [\frac{\rho}{2},\rho]$ and subsequently Poincar\'e's inequality. This leads to
\begin{align*}
	\mathrm{I}
	\le
	c\,\biint_{Q_\rho^{(\mu)}}
	\frac{|u-\langle u\rangle_{\rho}^{(\mu)}(t)|^q}{\mu^{\frac{q(2-p)}{p}}\rho^{q}}\,\dx\dt 
	\le
	c\,\biint_{Q_\rho^{(\mu)}}
	|Du|^q\,\dx\dt  ,
\end{align*}
for a constant $c=c(n,p)$. 
Now we will treat $\mathrm{II}$.
We start with an application of Lemma~\ref{lem:a-b} with $\alpha=p-1\ge1$
and subsequently the Gluing Lemma~\ref{lem:time-diff} with $R=\mu^{\frac{2-p}{p}}\rho$ and $S=\rho^p$. This gives
\begin{align}\label{est:II-pre}
    \mathrm{II}
    &\le
    \frac{\mu^{\frac{q(p-2)}{p}}}{\rho^q}
    \bint_{\Lambda_\rho}
    \big|\langle\power{u}{p{-}1}\rangle_{\hat\rho}^{(\mu)}(t)-
    (\power{u}{p{-}1})_{\hat\rho}^{(\mu)}
    \big|^\frac{q}{p-1}\dt\nonumber \\
    &\le
    \frac{\mu^{\frac{q(p-2)}{p}}}{\rho^q}
    \bint_{\Lambda_\rho}
    \bint_{\Lambda_\rho}
    \big|\langle\power{u}{p{-}1}\rangle_{\hat\rho}^{(\mu)}(t)-
    \langle\power{u}{p{-}1}\rangle_{\hat\rho}^{(\mu)}(\tau)
    \big|^\frac{q}{p-1}\dt\dtau \nonumber\\
	&\le
	c\,\mu^{\frac{q(p-2)}{p-1}} 
	\bigg[\biint_{Q_{\rho}^{(\mu)}} 
	\big[|Du|^{p-1} + |F|^{p-1}\big]\dx\dt\bigg]^\frac{q}{p-1} \nonumber\\
	&\le
	c\,\mu^{\frac{q(p-2)}{p-1}} 
	\biint_{Q_{\rho}^{(\mu)}} 
	\big[|Du|^{q} + |F|^{q}\big]\dx\dt .
\end{align}
If hypothesis \eqref{superintrinsic-mu-poin}$_2$ is satisfied, the lemma is proven. Therefore, it remains to consider the case where \eqref{superintrinsic-mu-poin}$_1$ is in force. Here, we apply 
Lemma~\ref{lem:alphalemma} to deduce that
\begin{align*}
     \bigg[\biint_{Q_\rho^{(\mu)}}
     |u|^p\,\dx\dt \bigg]^\frac1p
     &\le
     \bigg[\biint_{Q_\rho^{(\mu)}}
     \big|u-
     \power{\big[(u^{p-1})_{\hat\rho}^{(\mu)}\big]}{\frac{1}{p-1}}\big|^p\,\dx\dt
     \bigg]^{\frac{1}{p}}
     +
     \big|(\power{u}{p{-}1})_{\hat\rho}^{(\mu)}\big|^{\frac{1}{p-1}} \\
     &\le
     c\, \bigg[\biint_{Q_\rho^{(\mu)}}
     \big|u-(u)_{\rho}^{(\mu)}\big|^p\,\dx\dt \bigg]^{\frac{1}{p}}
     +
     \big|(\power{u}{p{-}1})_{\hat\rho}^{(\mu)}\big|^{\frac{1}{p-1}}  .
\end{align*}
Having arrived at this point, we take the last inequality to the power $q(p-2)$ and obtain
$$
	\mathrm{II}
	=
	\bigg[\biint_{Q_\rho^{(\mu)}} |u|^p\,\dx\dt \bigg]^\frac{q(p-2)}{p}
	\bigg[\biint_{Q_\rho^{(\mu)}} |u|^p\,\dx\dt \bigg]^{-\frac{q(p-2)}{p}}
	\cdot \mathrm{II}
	\le 
	c\big[\mathrm{II}_1+\mathrm{II}_2\big],
$$
with the abbreviations
\begin{align*}
     \mathrm{II}_1
     &:=
     \frac{\bigg[\displaystyle{\biint_{Q_\rho^{(\mu)}}
     \big|u-(u)_{\rho}^{(\mu)}\big|^p\,\dx\dt} \bigg]^{\frac{q(p-2)}{p}} }{
     \bigg[\displaystyle{\biint_{Q_\rho^{(\mu)}} |u|^p\,\dx\dt} \bigg]^{\frac{q(p-2)}{p}} }
     \cdot
     \mathrm{II},
\end{align*}
and 
\begin{align*}
    \mathrm{II}_2
    &:=
    \frac{\big|(\power{u}{p{-}1})_{\hat\rho}^{(\mu)}\big|^{\frac{q(p-2)}{p-1}}}{
	\bigg[\displaystyle{\biint_{Q_\rho^{(\mu)}} |u|^p\,\dx\dt} \bigg]^{\frac{q(p-2)}{p}}}
    \cdot
    \mathrm{II}.
\end{align*}
For the estimate of $\mathrm{II}_2$, we proceed as follows. We first insert the definition of $\mathrm{II}$, then use Lemma \ref{lem:Acerbi-Fusco} with $\alpha:=p-1\ge 1$ and finally apply the Gluing Lemma~\ref{lem:time-diff}. This  leads to
\begin{align*}
    \mathrm{II}_2
    &\le
    \frac{c\,\mu^{\frac{q(p-2)}{p}}}{
    \rho^{q}
    \displaystyle{\bigg[\biint_{Q_\rho^{(\mu)}} |u|^p\,\dx\dt \bigg]^{\frac{q(p-2)}{p}}}}\
    \bint_{\Lambda_\rho}
    \big|\langle\power{u}{p{-}1}\rangle_{\hat\rho}^{(\mu)}(t)-
    (\power{u}{p{-}1})_{\hat\rho}^{(\mu)}\big|^q
    \d t \\
    &\le
    \frac{c\,\mu^{\frac{2q(p-2)}{p}}}{
    \rho^{q(2-p)}
    \displaystyle{\bigg[\biint_{Q_\rho^{(\mu)}} |u|^p\,\dx\dt \bigg]^{\frac{q(p-2)}{p}}}}\
	\bigg[\biint_{Q_{\rho}^{(\mu)}} \big[|Du|^{p-1} + |F|^{p-1}\big]\dx\dt\bigg]^q .
\end{align*}
With H\"older's inequality and hypothesis \eqref{superintrinsic-mu-poin}$_1$, we finally obtain
\begin{align*}
    \mathrm{II}_2
    \le
    c\,\frac{\displaystyle{
    \bigg[\biint_{Q_{\rho}^{(\mu)}} \big[|Du|^{q} + |F|^{q}\big]\dx\dt\bigg]^{p-1}
    } }{
	\displaystyle{
	\bigg[\biint_{Q_{\rho}^{(\mu)}} \big[|Du|^{p} + |F|^{p}\big]\dx\dt\bigg]^{\frac{q(p-2)}{p}}
	} }
    \le
    c\,	\biint_{Q_{\rho}^{(\mu)}} \big[|Du|^{q} + |F|^{q}\big]\dx\dt .
\end{align*}
Inserting the preceding estimates above, we have shown that 
\begin{align}\label{poin-1}
     \biint_{Q_\rho^{(\mu)}}
     \frac{\big|u-(u)_{\rho}^{(\mu)}\big|^q}
     {\mu^{\frac{q(2-p)}{p}}\rho^q}\,\d x\d t
     \le 
     c\,\biint_{Q_\rho^{(\mu)}} \big[|Du|^{q} + |F|^{q}\big]\dx\dt +
     c\, \mathrm{II}_1.
\end{align}
For the estimate of $\mathrm{II}_1$, we use \eqref{est:II-pre} and hypothesis \eqref{superintrinsic-mu-poin}$_1$ to obtain
\begin{align}\label{II_1}
     \mathrm{II}_1
     &\le
     c\,\mu^{\frac{q(p-2)}{p-1}}
     \frac{\displaystyle{\bigg[\biint_{Q_\rho^{(\mu)}}
     \big|u-(u)_{\rho}^{(\mu)}\big|^p\,\dx\dt \bigg]^{\frac{q(p-2)}{p}}}}{ 
     \displaystyle{\bigg[\biint_{Q_\rho^{(\mu)}} |u|^p\,\dx\dt \bigg]^{\frac{q(p-2)}{p}}}}
     \biint_{Q_{\rho}^{(\mu)}} \big[|Du|^{q} + |F|^{q}\big]\dx\dt \nonumber\\
     &\le
     c\,\frac{\displaystyle{\bigg[\biint_{Q_\rho^{(\mu)}}
     \frac{\big|u-(u)_{\rho}^{(\mu)}\big|^p}{\mu^{2-p}\rho^p}\,\dx\dt \bigg]^{\frac{q(p-2)}{p}}}}{ 
     \displaystyle{\bigg[\biint_{Q_\rho^{(\mu)}} \frac{|u|^p}{\mu^{2-p}\rho^p}\,\dx\dt \bigg]^{\frac{q(p-2)^2}{p(p-1)}}}}
     \frac{\displaystyle{\biint_{Q_{\rho}^{(\mu)}} \big[|Du|^{q} + |F|^{q}\big]\dx\dt}}
     {\bigg[\displaystyle{\biint_{Q_{\rho}^{(\mu)}} \big[|Du|^{p} + |F|^{p}\big]\dx\dt}\bigg]^{\frac{q(p-2)}{p(p-1)}}} \nonumber\\
     &\le
     c\,\bigg[\biint_{Q_\rho^{(\mu)}}
     \frac{\big|u-(u)_{\rho}^{(\mu)}\big|^p}{\mu^{2-p}\rho^p}\,\dx\dt \bigg]^{\frac{q(p-2)}{p(p-1)}}
     \frac{\displaystyle{\biint_{Q_{\rho}^{(\mu)}} \big[|Du|^{q} + |F|^{q}\big]\dx\dt}}
     {\bigg[\displaystyle{\biint_{Q_{\rho}^{(\mu)}} \big[|Du|^{p} + |F|^{p}\big]\dx\dt}\bigg]^{\frac{q(p-2)}{p(p-1)}}}.
\end{align}
In the case $q=p$, we use Young's inequality with exponents $\frac{p-1}{p-2}$ and $p-1$ and obtain
\begin{align*}
     c\,\mathrm{II}_1
     \le
     \tfrac12
     \biint_{Q_\rho^{(\mu)}}
     \frac{\big|u-(u)_{\rho}^{(\mu)}\big|^p}
     {\mu^{2-p}\rho^{p}}\,\dx\dt +
     c\,\biint_{Q_{\rho}^{(\mu)}} \big[|Du|^{p} + |F|^{p}\big]\dx\dt.
\end{align*}
Inserting this into \eqref{poin-1} and reabsorbing the first term of the right-hand side into the left yields the desired Poincar\'e type inequality in the case $q=p$. At this point it remains to consider the case $q\in[p-1,p)$. Here, we use in \eqref{II_1} the Poincar\'e type inequality for $q=p$ to conclude
\begin{align*}
    \mathrm{II}_1
    &\le
    c\,
    \biint_{Q_{\rho}^{(\mu)}} \big[|Du|^{q} + |F|^{q}\big]\dx\dt.
\end{align*}
Together with \eqref{poin-1} this finishes the proof in the remaining case $q\in[p-1,p)$.
\end{proof}

As final result of this section we derive a Sobolev-Poincar\'e type inequality, which should be seen as the analogue of Lemma~\ref{lem:sob-1} for the superquadratic case.

\begin{lemma}\label{lem:sob-2}
Let $p\ge2$ and $u$ be a weak solution to \eqref{eq-doubly} in $\Omega_T$ in the sense of Definition~{\upshape\ref{def:weak_solution}}. 
Then, on any cylinder $Q_{\rho}^{(\mu)}(z_o)\subseteq\Omega_T$ with $\rho,\mu>0$ satisfying
\eqref{subintrinsic-mu-poin} and \eqref{superintrinsic-mu-poin} and for any $\epsilon\in (0,1]$, we have
\begin{align*}
    \biint_{Q_\rho^{(\mu)}(z_o)}
    \frac{\big|u-(u)_{z_o;\rho}^{(\mu)}\big|^p}
    {\mu^{2-p}\rho^p}\,\d x\d t
    &\le
	\epsilon\,\sup_{t\in \Lambda_\rho(t_o)}
	\bint_{B_\rho^{(\mu)}(x_o)}
	\frac{\big|\power{u}{\frac{p}{2}}(t)-
	\power{\big[(u)_{z_o;\rho}^{(\mu)}\big]}{\frac{p}{2}}\big|^2}{\rho^p}
	\dx \\
	&\quad+
	\frac{c}{\epsilon^{\frac{p-q}{2+q-p}}}
	\bigg[\biint_{Q_\rho^{(\mu)}(z_o)}
    \big[|Du|^{q} + |F|^{q}\big]
    \dx \dt\bigg]^{\frac{p}{q}} 
\end{align*}
with $q=\max\{\frac{np}{n+2},p{-}1\}$ and $c= c(n,p,L,K)$. 
\end{lemma}

\begin{proof}
As before, we omit the reference point $z_o$ in our
notation. Moreover, we abbreviate $a:=(u)_{\rho}^{(\mu)}$. Applying
Gagliardo-Nirenberg's inequality in Lemma~\ref{lem:gag} with $(p,q,r,\theta)$ replaced by
$(p,q,2,\frac{q}{p})$
and Lemma~\ref{lem:poin-2}, we find that
\begin{align}\label{gag-1}
    &\biint_{Q_\rho^{(\mu)}}
    \frac{|u-a|^p}{\mu^{2-p}\rho^p} \,\dx\dt \nonumber\\
	&\quad\le
	c \sup_{t\in \Lambda_\rho}
	\bigg[
	\bint_{B_\rho^{(\mu)}} \frac{|u(t)-a|^2}{\mu^{\frac{2(2-p)}{p}}\rho^2} \,\dx
	\bigg]^{\frac{p-q}{2}}
	\biint_{Q_\rho^{(\mu)}} 
	\bigg[
	\frac{|u-a|^q}{\mu^{\frac{q(2-p)}{p}}\rho^q} +
	|Du|^q \bigg] \,\dx \dt \nonumber\\
	&\quad\le
	c\sup_{t\in \Lambda_\rho}
	\bigg[\mu^{\frac{2(p-2)}{p}}
	\bint_{B_\rho^{(\mu)}} \frac{|u(t)-a|^2}{\rho^2} \,\dx
	\bigg]^{\frac{p-q}{2}}
	\biint_{Q_\rho^{(\mu)}} 
	\big[|Du|^q + |F|^q \big] \,\dx \dt .
\end{align}
We now exploit assumption \eqref{superintrinsic-mu-poin} in order to obtain an upper bound for $\mu$. If \eqref{superintrinsic-mu-poin}$_2$ is satisfied we have $\mu\le K^{1/p}$. On the other hand, if \eqref{superintrinsic-mu-poin}$_1$ is in force we apply Lemma~\ref{lem:poin-2} to infer that
\begin{align*}
	\bigg[\biint_{Q_{\rho}^{(\mu)}} \frac{\abs{u}^{p}}{\mu^{2-p}\rho^p}\dx\dt
	\bigg]^\frac1{p} 
	&\le
	\bigg[ 
	\biint_{Q_{\rho}^{(\mu)}} 
	\frac{|u - a|^{p}}{\mu^{2-p}\rho^p}\dx\dt\bigg]^\frac1{p}
	+
	\frac{\mu^{\frac{p-2}{p}}|a|}{\rho} \\
	&\le
	c\,\bigg[ 
    \biint_{Q_{\rho}^{(\mu)}} \big[|Du|^p + |F|^p\big] \dx\dt
    \bigg]^\frac1{p}
	+
	\frac{\mu^{\frac{p-2}{p}}|a|}{\rho} ,
\end{align*}
which in combination with the $\mu$-super-intrinsic coupling \eqref{superintrinsic-mu-poin}$_1$ yields
\begin{align*}
	\mu 
	&\le
	c
	+
	\frac{\mu^{\frac{p-2}{p}}|a|}
	{\rho\displaystyle{
	\bigg[\biint_{Q_{\rho}^{(\mu)}} \big[|Du|^p + |F|^p\big] \dx\dt
    \bigg]^{\frac1{p}}}} \\
	&\le
	\tfrac12\mu +
	c
	+
	\frac{|a|^{\frac{p}{2}}}
	{\rho^{\frac{p}{2}}\displaystyle{
	\bigg[\biint_{Q_{\rho}^{(\mu)}} \big[|Du|^p + |F|^p\big] \dx\dt
    \bigg]^{\frac1{2}}}} .
\end{align*}
This shows that
\begin{align*}
	\mu 
	\le
	c
	+
	\frac{2|a|^{\frac{p}{2}}}
	{\rho^{\frac{p}{2}}\displaystyle{
	\bigg[\biint_{Q_{\rho}^{(\mu)}} \big[|Du|^p + |F|^p\big] \dx\dt
    \bigg]^{\frac1{2}}}} 
\end{align*}
holds true in any case. Inserting this upper bound for $\mu$ into \eqref{gag-1} yields
\begin{align*}
    \biint_{Q_\rho^{(\mu)}}
    \frac{|u-a|^p}{\mu^{2-p}\rho^p} \,\dx\dt
    \le
    c\, \big[\mathrm{I}_1 + \mathrm{I}_2\big],
\end{align*}
with the obvious abbreviations 
\begin{align*}
    \mathrm{I}_1
    :=
	c\sup_{t\in \Lambda_\rho}
	\bigg[
	\bint_{B_\rho^{(\mu)}} \frac{|u(t)-a|^2}{\rho^2} \,\dx
	\bigg]^{\frac{p-q}{2}}
	\biint_{Q_\rho^{(\mu)}} 
	\big[|Du|^q + |F|^q \big] \,\dx \dt .
\end{align*}
and 
\begin{align*}
    \mathrm{I}_2
    :=
    \sup_{t\in \Lambda_\rho}
	\bigg[
	\bint_{B_\rho^{(\mu)}} 
	\frac{|a|^{p-2}|u(t)-a|^2}{\rho^p} \,\dx
	\bigg]^{\frac{p-q}{2}}
	\frac{\displaystyle{\biint_{Q_\rho^{(\mu)}} \big[|Du|^q + |F|^q \big] \,\dx\dt}}
	{\displaystyle{\bigg[\biint_{Q_{\rho}^{(\mu)}} \big[|Du|^p + |F|^p\big] \dx\dt
    \bigg]^{\frac{(p-2)(p-q)}{2p}}}} .
\end{align*}
For the first term, we use H\"older's inequality and Lemma~\ref{lem:a-b} to infer that
\begin{align*}
    \mathrm{I}_1
	&\le
	c\sup_{t\in \Lambda_\rho}
	\bigg[
	\bint_{B_\rho^{(\mu)}} \frac{|u(t)-a|^p}{\rho^p} \,\dx
	\bigg]^{\frac{p-q}{p}}
	\biint_{Q_\rho^{(\mu)}} 
	\big[|Du|^q + |F|^q \big] \,\dx \dt \\
	&\le
	c\sup_{t\in \Lambda_\rho}
	\bigg[
	\bint_{B_\rho^{(\mu)}} 
	\frac{\big|\power{u}{\frac p2}(t)-
	\power{a}{\frac{p}{2}}\big|^2}{\rho^p} 
	\,\dx
	\bigg]^{\frac{p-q}{p}}
	\biint_{Q_\rho^{(\mu)}} 
	\big[|Du|^q + |F|^q \big] \,\dx \dt  ,
\end{align*}
with a constant $c= c(n,p,L,K)$.
Now we turn our attention to the second term. With the help of
Lemma~\ref{lem:Acerbi-Fusco} applied with $\alpha=\frac p2$ and H\"older's inequality, we find that
\begin{align*}
    \mathrm{I}_2
    &\le
    c\sup_{t\in \Lambda_\rho}
	\bigg[
	\bint_{B_\rho^{(\mu)}} 
	\frac{\big|\power{u}{\frac p2}(t)-\power{a}{\frac p2}\big|^2}{\rho^p} \,\dx
	\bigg]^{\frac{p-q}{2}}
	\bigg[\biint_{Q_\rho^{(\mu)}} 
	\big[|Du|^q + |F|^q \big] \,\dx \dt \bigg]^{\frac{p(2+q-p)}{2q}} ,
\end{align*}
for a constant $c=c(n,p)$.
We add the resulting inequalities for $\mathrm{I}_1$ and $\mathrm{I}_2$ and apply Young's inequality. This yields the desired result.
\end{proof}

\section{Reverse H\"older inequality}

Our aim in this section is to derive a reverse H\"older type inequality for weak solutions of \eqref{eq-doubly}. It will be a consequence of the energy estimate in Lemma~\ref{lem:energy} and the Sobolev-Poincar\'e type inequality in Lemma~\ref{lem:sob-1}, respectively Lemma~\ref{lem:sob-2}. 

In contrast to Section~\ref{sec:poin} we now consider two concentric cylinders $Q_{\rho}^{(\mu)}(z_o)\subset Q_{2\rho}^{(\mu)}(z_o)\subseteq\Omega_T$ with $\rho,\mu>0$. We suppose that a $\mu$-sub-intrinsic coupling of the type
\begin{equation}\label{subintrinsic-mu}
	\frac{\displaystyle{\biint_{Q_{2\rho}^{(\mu)}(z_o)} 
	\frac{\abs{u}^{p}}{\mu^{2-\pf}(2\rho)^p}\dx\dt}}
	{\displaystyle{\biint_{Q_{\rho}^{(\mu)}(z_o)} \big[ |Du|^p +|F|^p\big]\dx\dt}}
	\le
	K\mu^{p}
\end{equation}
is satisfied for some $K\ge 1$. 
Furthermore, we assume that either 
\begin{equation}\label{superintrinsic-mu}
	\mu^{p}
	\le
	K
	\frac{\displaystyle{\biint_{Q_{\rho}^{(\mu)}(z_o)} 
	\frac{\abs{u}^{p}}{\mu^{2-\pf}\rho^p}\dx\dt}}{\displaystyle{\biint_{Q_{2\rho}^{(\mu)}(z_o)} \big[ |Du|^p +|F|^p\big]\dx\dt}}
	\qquad\mbox{or}\qquad
	\mu^{p}
	\le
	K 
\end{equation}
holds true. Then, we obtain the following reverse H\"older type inequality.

\begin{proposition}\label{prop:revhoelder-int}
Let $p>\max\{\frac{2n}{n+2},1\}$ and $u$ be a weak solution to \eqref{eq-doubly} in $\Omega_T$ in the sense of Definition~{\upshape\ref{def:weak_solution}}. 
Then, on any cylinder $Q_{2\rho}^{(\mu)}(z_o)\subseteq\Omega_T$ with $\rho,\mu>0$ satisfying~\eqref{subintrinsic-mu} and \eqref{superintrinsic-mu}, 
we have
\begin{align*}
	\biint_{Q_{\rho}^{(\mu)}(z_o)} |Du|^p \dx\dt 
	\le
	c\bigg[\biint_{Q_{2\rho}^{(\mu)}(z_o)} 
	|Du|^{q} \dx\dt \bigg]^{\frac{p}{q}} +
	c\, \biint_{Q_{2\rho}^{(\mu)}(z_o)} |F|^{p} \dx\dt 
\end{align*}
with the exponent $q:=\max\big\{\tfrac{2n}{n+2}, \frac{np}{n+2},1,p-1\big\}$
and a constant $c=c(n,p,\nu, L, K)$.
\end{proposition}

\begin{proof} 
Once again, we omit the reference to the center $z_o$ in the notation.
We consider radii $r,s$ with $\rho\le r<s\le 2\rho$ and let
\begin{equation}\label{R}
  \mathcal R_{r,s}
  :=
  \frac{s}{s-r}.
\end{equation}
Note that hypothesis \eqref{subintrinsic-mu} and \eqref{superintrinsic-mu} imply that the coupling conditions
\eqref{subintrinsic-mu-poin} and \eqref{superintrinsic-mu-poin} are
satisfied on $Q_s^{(\mu)}$ with the constant $2^{2n+3p}K$
in place of $K$. 
From now on we distinguish between the cases $\max\{\frac{2n}{n+2},1\}<p\le2$ and $p\ge2$. 

{\it The case $\max\{\frac{2n}{n+2},1\}<p\le2$.} 
Here the energy estimate from Lemma \ref{lem:energy} reads as  
\begin{align}\label{rev-h}
	\sup_{t \in \Lambda^{(\mu)}_r} &
	\mint_{B_r} 
	\frac{\big|\power{u}{\frac{p}{2}}(t) - \power{\big[(u)_r^{(\mu)}\big]}{\frac p2}\big|^2}{\mu^{p-2}r^p} \d x +
	\biint_{Q_r^{(\mu)}} |Du|^p \dx\dt \nonumber\\
	&\le
	c\,\biint_{Q_s^{(\mu)}} 
	\frac{\big|\power{u}{\frac p2} - \power{\big[(u)_r^{(\mu)}\big]}{\frac p2}\big|^2}
	{\mu^{p-2}(s^p-r^p)} \dx\dt
	+
	c\,\biint_{Q_s^{(\mu)}}
	\frac{\big|u-(u)_r^{(\mu)}\big|^p}{(s-r)^{p}}  \dx\dt \nonumber\\
	&\quad +
	c\, \biint_{Q_{s}^{(\mu)}} |F|^{p}\dx\dt \nonumber\\
	&=:
	\mathrm{I} + \mathrm{II} + \mathrm{III},
\end{align}
with the obvious meaning of I--III.  The constant $c$ depends only on $p,\nu$, and $L$.
We estimate $\mathrm{II}$ with the help of Lemma~\ref{lem:alphalemma}, Lemma~\ref{lem:2<1-1} and Young's inequality with the result that
\begin{align*}
	\mathrm{II}
	&\le
	c\,\mathcal R_{r,s}^{p}
	\biint_{Q_s^{(\mu)}} 
	\frac{\big|u-(u)_{s}^{(\mu)}\big|^{p}}{s^p}\, \dx\dt \\
	&\le
	c\,\mathcal R_{r,s}^{p}
	\bigg[\biint_{Q_{s}^{(\mu)}} \big[ |Du|^p +|F|^p\big]\dx\dt
	\bigg]^\frac{2-p}{2}
	\Bigg[
    \biint_{Q_s^{(\mu)}}
    \frac{\big|\power{u}{\frac{p}{2}}-
    \power{\big[(u)_{s}^{(\mu)}\big]}{\frac p2}\big|^2}
    {\mu^{p-2}s^p}\,\dx\dt
    \Bigg]^{\frac p2} \\
    &\le
    c\,\mathcal R_{r,s}^{p} \Bigg[
    \delta\, \biint_{Q_{s}^{(\mu)}} \big[ |Du|^p +|F|^p\big]\dx\dt +
    \frac{1}{\delta^{\frac{2-p}{p}}} 
    \biint_{Q_s^{(\mu)}}
    \frac{\big|\power{u}{\frac{p}{2}}-
    \power{\big[(u)_{s}^{(\mu)}\big]}{\frac p2}\big|^2}
    {\mu^{p-2}s^p}\,\dx\dt \Bigg]
\end{align*}
holds true for any $\delta\in(0,1]$. 
Taking into account that $(s-r)^p\le s^p-r^p$, we obtain due to Lemma~\ref{lem:alphalemma} that 
\begin{align*}
	\mathrm{I}
	&\le
	c\,\mathcal R_{r,s}^{p}
	\biint_{Q_s^{(\mu)}}
	\frac{\big|\power{u}{\frac p2} - \power{\big[(u)_s^{(\mu)}\big]}{\frac p2}\big|^2}
	{\mu^{p-2}s^{p}} \dx\dt . 
\end{align*}
We add both inequalities and apply Lemma~\ref{lem:sob-1} on $Q_s^{(\mu)}$ with $\epsilon=\delta^{\frac2p}$. In this way we obtain
\begin{align*}
	\mathrm{I} + \mathrm{II}
    &\le
    c\,\mathcal R_{r,s}^{p} \Bigg[
    \delta\, \biint_{Q_{s}^{(\mu)}} \big[ |Du|^p +|F|^p\big]\dx\dt +
    \frac{1}{\delta^{\frac{2-p}{p}}} 
    \biint_{Q_s^{(\mu)}}
    \frac{\big|\power{u}{\frac{p}{2}}-
    \power{\big[(u)_{s}^{(\mu)}\big]}{\frac p2}\big|^2}
    {\mu^{p-2}s^p}\,\dx\dt \Bigg] \\
    &\le
	c\,\delta\mathcal R_{r,s}^{p}\Bigg[
	\sup_{t\in \Lambda_s^{(\mu)}}
	\bint_{B_s}
	\frac{\big|\power{u}{\frac{p}{2}}(t)-
	\power{\big[(u)_{s}^{(\mu)}\big]}{\frac{p}{2}}\big|^2}{\mu^{p-2}s^p}
	\dx +
	\biint_{Q_{s}^{(\mu)}} |Du|^p \dx\dt\Bigg]\\
	&\quad+
	\frac{c\,\mathcal R_{r,s}^{p}}{\delta^{\frac{4}{pq(p-1)}-1}}
	\Bigg[
	\bigg[\biint_{Q_s^{(\mu)}} |Du|^{q} \dx\dt\bigg]^{\frac{p}{q}} +
	\biint_{Q_s^{(\mu)}} |F|^{q} \dx\dt 
	\Bigg],
\end{align*}
where $q=\max\{\frac{2n}{n+2},1\}$. We insert this inequality into \eqref{rev-h}. Then, we choose 
\begin{equation}\label{choice-delta}
	\delta=\frac{1}{2c\, \mathcal R_{r,s}^{p}}
\end{equation}
and apply the Iteration Lemma \ref{lem:tech} to re-absorb the term $\frac12[\dots]$ from the right-hand side into the left. This leads to the claimed reverse H\"older type inequality, i.e.~to
\begin{align*}
	\sup_{t \in \Lambda^{(\mu)}_\rho}
	\mint_{B_\rho} &
	\frac{\big|\power{u}{\frac{p}{2}}(t) - \power{\big[(u)_\rho^{(\mu)}\big]}{\frac p2}\big|^2}{\mu^{p-2}\rho^p} \d x +
	\biint_{Q_\rho^{(\mu)}} |Du|^p \dx\dt \\
	&\le
	c\,\Bigg[\bigg[\biint_{Q_{2\rho}^{(\mu)}}
	|Du|^{q} \dx\dt \bigg]^{\frac{p}{q}} +
	 \biint_{Q_{2\rho}^{(\mu)}}|F|^{p} \dx\dt\Bigg] 
\end{align*}
and finishes the proof of Proposition \ref{prop:revhoelder-int} in the case $\max\{\frac{2n}{n+2},1\}<p\le 2$.

{\it The case $p\ge2$.}
In this case, the energy estimate from Lemma \ref{lem:energy} yields 
\begin{align}\label{rev-h-2}
	\sup_{t \in \Lambda_r} &
	\mint_{B_r^{(\mu)}} 
	\frac{\big|\power{u}{\frac{p}{2}}(t) - \power{\big[(u)_r^{(\mu)}\big]}{\frac p2}\big|^2}{r^p} \d x +
	\biint_{Q_r^{(\mu)}} |Du|^p \dx\dt \nonumber\\
	&\le
	c\,\biint_{Q_s^{(\mu)}} 
	\frac{\big|\power{u}{\frac p2} - \power{\big[(u)_r^{(\mu)}\big]}{\frac p2}\big|^2}
	{s^p-r^p} \dx\dt
	+
	c\,\biint_{Q_s^{(\mu)}}
	\frac{\big|u-(u)_r^{(\mu)}\big|^p}{\mu^{2-p}(s-r)^{p}}  \dx\dt \nonumber\\
	&\quad +
	c\, \biint_{Q_{s}^{(\mu)}} |F|^{p}\dx\dt \nonumber\\
	&=:
	\mathrm{I} + \mathrm{II} + \mathrm{III}
\end{align}
with the obvious meaning of $\mathrm{I}$--$\mathrm{III}$. 
Now, we estimate the term $\mathrm{I}$ by using the fact that $(s-r)^p\le
s^p-r^p$, Lemma~\ref{lem:alphalemma}, Lemma~\ref{lem:1<2-2} and Young's inequality. 
In this way we obtain 
\begin{align*}
  \mathrm{I}
  &\le
  c\,\mathcal R_{r,s}^{p}
  \biint_{Q_s^{(\mu)}}
  \frac{\big|\power{u}{\frac{p}{2}}-
  \power{\big[(u)_{s}^{(\mu)}\big]}{\frac p2}\big|^2}
  {s^p}\,\dx\dt\\
  &\le
  c\,\mathcal R_{r,s}^{p}
  \bigg[\biint_{Q_{s}^{(\mu)}} \big[ |Du|^p +|F|^p\big]\dx\dt
  \bigg]^\frac{p-2}{p}
  \bigg[
	\biint_{Q_s^{(\mu)}} 
	\frac{\big|u-(u)_{s}^{(\mu)}\big|^{p}}{\mu^{2-p}s^p}\, \dx\dt
	\bigg]^\frac2p \\
    &\le
    c\,\mathcal R_{r,s}^{p} \Bigg[
    \delta\, \biint_{Q_{s}^{(\mu)}} \big[ |Du|^p +|F|^p\big]\dx\dt +
    \frac{1}{\delta^{\frac{p-2}2}} 
    \biint_{Q_s^{(\mu)}} 
    \frac{\big|u-(u)_{s}^{(\mu)}\big|^{p}}{\mu^{2-p}s^p}\, \dx\dt \Bigg],
\end{align*}
for any $\delta\in(0,1]$. 
Moreover, from Lemma~\ref{lem:alphalemma} we know that 
\begin{align*}
	\mathrm{II}
	&\le
	c\,\mathcal R_{r,s}^{p}
	\biint_{Q_s^{(\mu)}} 
	\frac{\big|u-(u)_{s}^{(\mu)}\big|^{p}}{\mu^{2-p}s^p}\, \dx\dt .
\end{align*}
We combine the preceding estimates and apply Lemma~\ref{lem:sob-2}
with $\eps=\delta^{\frac p2}$ in order to obtain
\begin{align*}
  	\mathrm{I} + \mathrm{II}
  	&\le
  	c\,\mathcal R_{r,s}^{p} \Bigg[
  	\delta\, \biint_{Q_{s}^{(\mu)}} \big[ |Du|^p +|F|^p\big]\dx\dt +
  	\frac{1}{\delta^{\frac{p-2}2}} 
  	\biint_{Q_s^{(\mu)}} 
  	\frac{\big|u-(u)_{s}^{(\mu)}\big|^{p}}{\mu^{2-p}s^p}\, \dx\dt \Bigg]\\
  	&\le
  	c\,\delta\,\mathcal R_{r,s}^{p} \Bigg[
	\sup_{t\in \Lambda_s}
    \bint_{B_s^{(\mu)}}
	\frac{\big|\power{u}{\frac{p}{2}}(t)-
	\power{\big[(u)_{s}^{(\mu)}\big]}{\frac{p}{2}}\big|^2}{s^p}
	\dx +
    \biint_{Q_{s}^{(\mu)}} |Du|^p \dx\dt \Bigg] \\
	&\quad+
	\frac{c\,\mathcal R_{r,s}^{p}}{\delta^{\frac{p}{2+p-q}-1}}
	\Bigg[
	\bigg[\biint_{Q_s^{(\mu)}} |Du|^{q} \dx\dt\bigg]^{\frac{p}{q}} +
	\biint_{Q_s^{(\mu)}} |F|^{p} \dx\dt
	\Bigg] ,
\end{align*}
where $q=\max\{\frac{np}{n+2},p{-}1\}$. 
As before, we insert this inequality into \eqref{rev-h-2}, choose $\delta\in(0,1]$ of the form \eqref{choice-delta}
and apply the Iteration Lemma~\ref{lem:tech}.          
This allows to re-absorb the term $\frac12[\dots]$ into the left-hand side and yields the desired reverse H\"older type inequality in the remaining case $p\ge 2$. 
This finishes the proof of the proposition.
\end{proof}

\section{Higher integrability: Proof of Theorem \ref{thm:higherint}}\label{sec:hi}

In this section we finally prove the higher integrability result of Theorem \ref{thm:higherint}. 
We consider a fixed cylinder $Q_{4R}\equiv Q_{4R,(4R)^p}(\mathfrak z_o)\subset\Omega_T$
with $R>0$ and let
\begin{equation}\label{first-lambda-0}
	\lambda
	\ge
	\lambda_o
	\ge
	1+\bigg[\biint_{Q_{4R}} 
	\frac{\abs{u}^{p}}{(4R)^p} \d x\d t\bigg]^{\frac{1}{p}}.
\end{equation}
We recall the notation \eqref{cylinder} for the scaled cylinders $Q_\rho^{(\mu)}(z_o)$, and observe that $Q_\rho^{(\mu)}(z_o)\subset Q_\rho^{(\kappa)}(z_o)$ whenever $\kappa\le\mu$. Moreover, we recall the abbreviation $\pf=\max\{2,p\}$.

\subsection{Construction of a non-uniform system of cylinders}\label{sec:cylinders}
The main difficulty now is to construct a covering of the
$\lambda$-superlevel set of $|Du|$ by cylinders on which the reverse
H\"older type inequality from Proposition~\ref{prop:revhoelder-int} is
applicable. This means that the scaled cylinders have to satisfy
hypothesis \eqref{subintrinsic-mu} and \eqref{superintrinsic-mu}. The
following construction of a non-uniform system of cylinders is
inspired by the one in \cite{Schwarzacher, Gianazza-Schwarzacher}.  Let $z_o\in Q_{2R}$. For a radius $\rho\in (0,R]$ we now define
$$
	\widetilde\mu_{z_o;\rho}^{(\lambda)}
	:=
	\inf\bigg\{\mu\in[1,\infty):
	\frac{1}{|Q_\rho|}
	\iint_{Q_{\rho}^{(\mu)}(z_o)}
	\frac{\abs{u}^{p}}{\rho^p} \dx\dt 
	\le 
	\mu^{p-\p}\lambda^{p} \bigg\} ,
$$
where $\p:=2-p+(\pf-2)(2+\frac{n}{p})$. 
Note that 
$$
	\mu^{p-\p}
	=
	\left\{\begin{array}{ll}
	\mu^{2(p-1)}& \quad \mbox{if $p\le 2$,}\\[5pt]
	\mu^\frac{2n-p(n-2)}{p}& \quad \mbox{if $p> 2$.}
	\end{array}\right.
$$
In particular, the restriction $p<\frac{2n}{n-2}$ for $n>2$ ensures
that $p-\p>0$ in any case.
However, we note that in dimensions $n>2$,
the exponent of $\mu$ tends to zero in the limit $p\uparrow \frac{2n}{n-2}$. This is the only point where the restriction $p<\frac{2n}{n-2}$ enters the proof. 
If $z_o$ and $\lambda$ are fixed and if the meaning is clear from the context we write $\widetilde\mu_\rho$ instead of
$\widetilde\mu_{z_o;\rho}^{(\lambda)}$. 
Observe that the set of  those $\mu\ge1$ for which the condition in the infimum is satisfied is not empty.
In fact, in the limit $\mu\uparrow\infty$ the integral on the left-hand side converges to zero
(note that the measure of $Q_{\rho}^{(\mu)}(z_o)$ shrinks to 0),
while the right-hand side blows up with speed $\mu^{p-\p}$ (recall
that $p-\p>0$). We point out that the condition 
in the infimum is equivalent to
$$
	\biint_{Q_{\rho}^{(\mu)}(z_o)}
	\frac{\abs{u}^{p}}{\mu^{2-\pf}\rho^p} \dx\dt 
	\le 
	\mu^{p}\lambda^p .
$$
Therefore, we either  have
$$
	\widetilde\mu_\rho=1
	\qquad\mbox{and}\qquad
	\biint_{Q_{\rho}^{(\widetilde\mu_\rho)}(z_o)} 
	\frac{\abs{u}^{p}}{\widetilde\mu_\rho^{2-\pf}\rho^p} \dx\dt
	\le
	\widetilde\mu_\rho^{p}\lambda^p
	=
	\lambda^{p},
$$
or otherwise
\begin{equation}\label{theta>lambda}
	\widetilde\mu_\rho>1
	\qquad\mbox{and}\qquad
	\biint_{Q_{\rho}^{(\widetilde\mu_\rho)}(z_o)} 
	\frac{\abs{u}^{p}}{\widetilde\mu_\rho^{2-\pf}\rho^p} \dx\dt
	=
	\widetilde\mu_\rho^{p}\lambda^p .
\end{equation}
Using this observation for $\rho=R$, we have that either $\widetilde \mu_{R}= 1$, or $\widetilde\mu_{R}>1$ and 
\begin{align*}
	\widetilde\mu_{R}^{p-\p}
	=
	\frac{1}{\lambda^{p}|Q_R|}
	\iint_{Q_{R}^{(\widetilde\mu_R)}(z_o)}
	\frac{\abs{u}^{p}}{R^p} \dx\dt 
	\le
	\frac{1}{\lambda^p}
	\biint_{Q_{R}(z_o)} \frac{\abs{u}^{p}}{R^p}\dx\dt 
	\le
	\frac{4^{n+2p}\lambda_o^{p}}{\lambda^p}
	\le 
	4^{n+2p} .
\end{align*}
Therefore, in any case we have the bound
\begin{align}\label{bound-theta-R}
	\widetilde\mu_{R} 
	\le
    4^{\frac{n+2p}{p-\p}}.
\end{align}
Our next aim is to ensure that the mapping $(0,R]\ni\rho\mapsto
\widetilde\mu_\rho$ is continuous.
To this end, we consider $\rho\in(0,R]$ and  $\eps>0$, and define
$\mu_+:=\widetilde\mu_\rho+\eps$. Then,  
there exists $\delta=\delta(\eps,\rho)>0$ such that
\begin{equation*}
 	 \frac{1}{|Q_r|}\iint_{Q_{r}^{(\mu_+)}(z_o)} 
  	 \frac{\abs{u}^{p}}{r^p} \dx\dt 
  	<
  	\mu_{+}^{p-\p}\lambda^{p}
\end{equation*}
for any $r\in(0,R]$ with $|r-\rho|<\delta$. In fact, due to the definition of $\widetilde\mu_\rho$ the preceding strict inequality holds for $r=\rho$, since $\mu_+>\widetilde\mu_\rho$ and $Q_{\rho}^{(\mu_+)}(z_o) \subset Q_{\rho}^{(\widetilde\mu_\rho)}(z_o)$. 
The claim now follows, since the left-hand side depends continuously on the radius $r$. Recalling the very definition of $\widetilde\mu_r$, the last inequality implies $\widetilde\mu_r\le\mu_+=\widetilde\mu_\rho+\eps$ for any $r\in(0,R]$ with $|r-\rho|<\delta$. 
It remains to prove $\widetilde\mu_r\ge\mu_-:=\widetilde\mu_\rho-\eps$ for $r$ close to $\rho$. 
If $\widetilde\mu_\rho=1$, then we have $\widetilde\mu_r\ge 1 =\widetilde\mu_\rho\ge \mu_-$. 
If $\widetilde\mu_\rho>1$, we get after diminishing
$\delta=\delta(\eps,\rho)>0$ if necessary that 
\begin{equation*}
	\frac{1}{|Q_r|}
  	\iint_{Q_{r}^{(\mu_-)}(z_o)}
  	\frac{\abs{u}^{p}}{r^p} \dx\dt 
  	>
  	\mu_-^{p-\p} \lambda^p
\end{equation*}
for all $r\in(0,R]$ with $|r-\rho|<\delta$. For $r=\rho$, this is a direct consequence of the definition of $\widetilde\mu_\rho$, since $\widetilde\mu_\rho>\mu_-$ and 
$Q_{\rho}^{(\widetilde\mu_\rho)}(z_o)\subset Q_{\rho}^{(\mu_-)}(z_o)$.
Due to the continuity of the left-hand side with respect to $r$, this implies the claim for $r$ with $|r-\rho|<\delta$ small enough. 
The preceding inequality implies that $\widetilde\mu_r\ge\mu_-:=\widetilde\mu_\rho-\eps$. 
This completes the proof of the continuity of $(0,R]\ni\rho\mapsto
\widetilde\mu_\rho$.

Unfortunately, the mapping $\rho\mapsto\widetilde\mu_\rho$ might not be
monotone. For this reason we modify $\widetilde\mu_\rho$ in such a way that the modification  -- denoted by $\mu_\rho$ -- becomes monotone. Therefore, we define
$$
	\mu_\rho
	\equiv
	\mu_{z_o;\rho}^{(\lambda)}
	:=
	\max_{r\in[\rho,R]} \widetilde\mu_{z_o;r}^{(\lambda)}. 
$$
As before, we abbreviate $\mu_{z_o;\rho}^{(\lambda)}$ by $\mu_\rho$ 
if $z_o$ and $\lambda$ are fixed, so that no confusion is possible.
By construction the mapping
$(0,R]\ni\rho\mapsto \mu_\rho$ is
continuous and monotonically decreasing. 
Moreover, the cylinders $Q_{s}^{(\mu_\rho)}(z_o)$ are $\mu$-sub-intrinsic (with constant $K =1$) whenever $\rho\le s$. More precisely, we have\footnote{
Note that later 
$\lambda^p\approx\displaystyle{\biint_{Q_{s}^{(\mu_{\rho})}(z_o)}\big[|Du|^p+|F|^p\big]\dx\dt}$.  This justifies the notion $\mu$-sub-intrinsic in the sense of \eqref{subintrinsic-mu}.} 
\begin{align}\label{sub-intrinsic}
	\biint_{Q_{s}^{(\mu_{\rho})}(z_o)} 
	\frac{\abs{u}^{p}}{\mu_\rho^{2-\pf}s^p} \dx\dt
	\le 
	\mu_\rho^{p}\lambda^p
	\quad\mbox{for any $0<\rho\le s\le R$.}
\end{align}
In fact, the definition of $\mu_s$ and the   monotonicity of $\mu_\rho$
imply $\widetilde\mu_s\le \mu_{s}\le \mu_{\rho}$, so that 
$Q_{s}^{(\mu_{\rho})}(z_o)\subset Q_{s}^{(\widetilde\mu_{s})}(z_o)$. This allows to estimate 
\begin{align*}
	\biint_{Q_{s}^{(\mu_{\rho})}(z_o)} 
	\frac{\abs{u}^{p}}{\mu_\rho^{2-\pf}s^p} \dx\dt
	&\le 
	\Big(\frac{\mu_{\rho}}{\widetilde\mu_{s}}\Big)^{\p}
	\biint_{Q_{s}^{(\widetilde\mu_{s})}(z_o)} 
	\frac{\abs{u}^{p}}{\widetilde\mu_s^{2-\pf}s^p} \dx\dt \\
	&\le 
	\Big(\frac{\mu_{\rho}}{\widetilde\mu_{s}}\Big)^{\p}
	\widetilde\mu_{s}^{p} \lambda^p 
	\le
	\mu_\rho^{\p}\,
	\widetilde\mu_s^{p-\p}
	\lambda^p
	\le 
	\mu_\rho^{p} \lambda^p.
\end{align*}
In the last step we used the fact 
$p-\p>0$. We now define 
\begin{equation}\label{rho-tilde}
	\widetilde\rho
	:=
	\left\{
	\begin{array}{cl}
	R &
	\quad\mbox{if $\mu_\rho=1$,} \\[4pt]
	\inf\big\{s\in[\rho, R]: \mu_s=\widetilde\mu_s \big\} &
	\quad\mbox{if $\mu_\rho>1$.}
	\end{array}
	\right.
\end{equation}
Note that 
$\mu_s=\widetilde\mu_{\widetilde\rho}$ for any $s\in [\rho,\widetilde\rho]$ and in particular $\mu_{\rho}=\widetilde\mu_{\widetilde\rho}$.
Next, we claim that 
\begin{align}\label{bound-theta}
	\mu_\rho 
	\le
	\Big(\frac{s}{\rho}\Big)^{\frac{n+2p}{p-\p}}
	\mu_{s} 
	\quad\mbox{for any $s\in(\rho,R]$.}
\end{align}
If $\mu_\rho=1$, then  also $\mu_s=1$, so that
\eqref{bound-theta} trivially holds.
Therefore, it remains to consider the case $\mu_\rho >1$. If $s\in(\rho,\widetilde\rho]$, then $\mu_\rho=\mu_{s}$, and 
\eqref{bound-theta} obviously holds true. 
Otherwise, if $s\in(\widetilde\rho,R]$, then \eqref{theta>lambda}, the monotonicity of $s\mapsto\mu_s$ and
\eqref{sub-intrinsic}  imply 
\begin{align*}
	\mu_\rho^{p-\p} 
	&=
	\widetilde \mu_{\widetilde\rho}^{p-\p}
	=
	\frac{1}{\lambda^p|Q_{\widetilde\rho}|} 
	\iint_{Q_{\widetilde\rho}^{(\mu_{\widetilde\rho})}(z_o)}
	\frac{\abs{u}^{p}}{\widetilde\rho^p} \dx\dt \\
	&\le
	\Big(\frac{s}{\widetilde\rho}\Big)^{n+2p}
	\frac{1}{\lambda^p|Q_s|} 
	\iint_{Q_{s}^{(\mu_s)}(z_o)} 
	\frac{\abs{u}^{p}}{s^p} \dx\dt 
	\le
	\Big(\frac{s}{\rho}\Big)^{n+2p}
	\mu_{s}^{p-\p} .
\end{align*}
This proves the claim \eqref{bound-theta}.  
We now apply \eqref{bound-theta} with $s=R$. Since  $\mu_{R}=\widetilde\mu_{R}$, the bound \eqref{bound-theta-R} for $\widetilde\mu_{R}$ yields
\begin{align}\label{bound-theta-2}
	\mu_\rho 
	\le
	\Big(\frac{R}{\rho}\Big)^{\frac{n+2p}{p-\p}}
	\mu_{R} 
	\le
	\Big(\frac{4R}{\rho}\Big)^{\frac{n+2p}{p-\p}} .
\end{align}
In the following, we consider the system of concentric cylinders
$Q_{\rho}^{(\mu_{z_o;\rho}^{(\lambda)})}(z_o)$ with radii $\rho\in (0,R]$ and
$z_o\in Q_{2R}$. The cylinders are nested, in the sense that
$$
	Q_{r}^{(\mu_{z_o;r}^{(\lambda)})}(z_o)
	\subset
	Q_{s}^{(\mu_{z_o;s}^{(\lambda)})}(z_o)
	\subset
	Q_{4R}
	\mbox{ whenever $0<r<s\le R$.}
$$
The inclusions hold true due to the monotonicity of the mapping $\rho\mapsto \mu_{z_o;\rho}^{(\lambda)}$ and the fact that $\mu_{z_o;\rho}^{(\lambda)}\ge 1$.
The disadvantage of using $\mu_{z_o;\rho}^{(\lambda)}$ instead of
$\widetilde\mu_{z_o;\rho}^{(\lambda)}$ is that the associated
cylinders are in general only $\mu$-sub-intrinsic with $K =1$, but
not $\mu$-intrinsic.

\subsection{Covering property}\label{sec:covering}
The system of cylinders $Q_{r}^{(\mu_{z_o;r}^{(\lambda)})}(z_o)$ constructed above satisfies a Vitali type covering property. This will be proven in the following lemma. 

\begin{lemma}\label{lem:vitali}
There exists a constant $\hat c=\hat c(n,p)\ge 20$ such that whenever $\lambda\ge\lambda_o$ and $\mathcal F$ is any collection of cylinders $Q_{4r}^{(\mu_{z;r}^{(\lambda)})}(z)$, where $Q_{r}^{(\mu_{z;r}^{(\lambda)})}(z)$ is a cylinder of the form as constructed in Section {\upshape\ref{sec:cylinders}} with radius $r\in(0,\tfrac{R}{\hat c}]$, then there exists a countable subfamily $\mathcal G$ of disjoint cylinders in $\mathcal F$ such that  
\begin{equation}\label{covering}
	\bigcup_{Q\in\mathcal F} Q
	\subset 
	\bigcup_{Q\in\mathcal G} \widehat Q,
\end{equation}
where $\widehat Q$ denotes the $\frac{1}{4}\hat c$-times enlarged cylinder $Q$, i.e.~if $Q=Q_{4r}^{(\mu_{z;r}^{(\lambda)})}(z)$, then $\widehat Q=Q_{\hat c r}^{(\mu_{z;r}^{(\lambda)})}(z)$.
\end{lemma}

\begin{proof}
Throughout the proof we abbreviate
$\mu_{z;r}:=\mu_{z;r}^{(\lambda)}$. We let $\hat c\ge 20$ be a parameter that will be chosen later. 
For $j\in \N$ we define
$$
	\mathcal F_j
	:=
	\Big\{Q_{4r}^{(\mu_{z;r})}(z)\in \mathcal F: 
	\tfrac{R}{2^j\hat c}<r\le \tfrac{R}{2^{j-1}\hat c} \Big\}
$$
and select $\mathcal G_j\subset \mathcal F_j$ by the following procedure:
We choose $\mathcal G_1$ to be any maximal disjoint collection of cylinders in
$\mathcal F_1$. Note that $\mathcal G_1$ contains only finitely many cylinders, since by  the definition of $\mathcal F_1$ and \eqref{bound-theta-2} the $\mathcal L^{n+1}$-measure of each cylinder $Q\in \mathcal G_1$ is uniformly bounded from below. Now, assume that for some $k\in\N_{\ge2}$ the collections $\mathcal G_1, \mathcal G_2, \dots, \mathcal G_{k-1}$ have already been inductively selected. Then, we choose a maximal  disjoint sub-collection of cylinders from $\mathcal F_k$ which do not intersect any of the cylinders $Q^\ast$ from one of the collections $\mathcal G_j$, $j\in\{ 1,\dots ,k-1\}$. More precisely, we choose a maximal disjoint collection of cylinders in
$$
	\bigg\{Q\in \mathcal F_k: 
	Q\cap Q^\ast=\emptyset 
	\mbox{ for any $ \displaystyle Q^\ast\in \bigcup_{j=1}^{k-1} \mathcal G_j $}
	\bigg\}.
$$
Note again that $\mathcal G_k$ is finite.
Finally, we let 
$$
	\mathcal G
	:=
	\bigcup_{j=1}^\infty \mathcal G_j.
$$
By construction, $\mathcal G\subset\mathcal F$ is a countable subfamily of disjoint cylinders in $\mathcal F$.

At this point it remains to prove that for each $Q\in\mathcal F$ there exists a cylinder $Q^\ast\in\mathcal G$ such that $Q\cap Q^\ast\not=\emptyset$ and $Q\subset \widehat {Q}^\ast$.  
To this aim we consider some arbitrary cylinder $Q=Q_{4r}^{(\mu_{z;r})}(z)\in\mathcal F$. Then, there exists an index $j\in\N$ such that $Q\in\mathcal F_j$. The maximality of $\mathcal G_j$ ensures that there exists a cylinder $Q^\ast=Q_{4r_\ast}^{(\mu_{z_\ast;r_\ast})}(z_\ast)\in \bigcup_{i=1}^{j} \mathcal G_i$ with $Q\cap Q^\ast\not=\emptyset$. Then, we have $r<2r_\ast$, since $r\le\tfrac{R}{2^{j-1}\hat c}$ and $r_\ast>\tfrac{R}{2^j\hat c}$.
The main difficulty now is to establish a bound for $\mu_{z_\ast;r_\ast}$ in terms of $\mu_{z;r}$. We claim that the following estimate holds true:
\begin{equation}\label{control-mu-1}
	\mu_{z_\ast;r_\ast}
	\le
	(4\eta)^{\frac{n+2p}{p-\p}}\,
	\mu_{z;r} ,
\end{equation}
where $\eta:= 13$. To prove the claim we denote by $\widetilde
r_\ast\in [r_\ast,R]$ the radius associated to the cylinder
$Q_{r_\ast}^{(\mu_{z_\ast;r_\ast})}(z_\ast)$; see  \eqref{rho-tilde} for the construction. Recall
that either $\mu_{z_\ast;r_\ast}=1$ and $\widetilde r_\ast=R$ or $Q_{\widetilde r_\ast}^{(\mu_{z_\ast;r_\ast})}(z_\ast)$ is
intrinsic in the sense of~\eqref{theta>lambda}. In the former case we have 
$$
	\mu_{z_\ast;r_\ast}
	=
	1
	\le 
	\mu_{z;r},
$$
so that \eqref{control-mu-1} is satisfied. 
If $Q_{\widetilde r_\ast}^{(\mu_{z_\ast;r_\ast})}(z_\ast)$ is intrinsic, we know that
\begin{align}\label{control-mu-2}
	\mu_{z_\ast;r_\ast}^{p-\p}
	=
	\frac{1}{\lambda^p|Q_{\widetilde r_\ast}|}
	\iint_{Q_{\widetilde r_\ast}^{(\mu_{z_\ast;r_\ast})}(z_\ast)}
	\frac{\abs{u}^p}{\widetilde r_\ast^p} \dx\dt. 
\end{align}
Now, we distinguish between the cases  $\widetilde r_\ast\le \frac{R}{\eta}$ and $\widetilde r_\ast> \frac{R}{\eta}$.
We first consider the simpler case $\widetilde r_\ast> \frac{R}{\eta}$. Here, we exploit \eqref{control-mu-2} and \eqref{first-lambda-0} to conclude that 
\begin{align*}
  	\mu_{z_\ast;r_\ast}^{p-\p}
	\le
	\Big(\frac{4R}{\widetilde r_\ast}\Big)^{p} 
	\frac{1}{\lambda^p|Q_{\widetilde r_\ast}|}
	\iint_{Q_{4R}}
	\frac{\abs{u}^p}{(4R)^p} \dx\dt 
	\le
	\Big(\frac{4R}{\widetilde r_\ast}\Big)^{p} 
	\frac{|Q_{4R}|}{|Q_{\widetilde r_\ast}|}
	\le
	(4\eta)^{n+2p} ,
\end{align*}
which implies
\begin{align*}
	\mu_{z_\ast;r_\ast}
	\le
	(4\eta)^{\frac{n+2p}{p-\p}}
	\le
	(4\eta)^{\frac{n+2p}{p-\p}}
	\mu_{z;r}
\end{align*}
and proves \eqref{control-mu-1} in this case. 
Therefore it remains to consider radii $\widetilde r_\ast\le \frac{R}{\eta}$. 
Note that we can assume 
$\mu_{z;r}\le\mu_{z_\ast;r_\ast}$. Otherwise 
\eqref{control-mu-1} trivially holds.
Therefore, the monotonicity of $\rho\mapsto \mu_{z;\rho}$ and the fact that $r\le 2r_\ast\le 2 \widetilde r_\ast\le\eta\widetilde r_\ast$ imply
\begin{equation}\label{mu-ast}
	\mu_{z;\eta \widetilde r_\ast}
	\le
	\mu_{z;r}
	\le
	\mu_{z_\ast;r_\ast}.
\end{equation}
Next, we claim that 
\begin{equation}\label{inclusion-eta}
	Q_{\widetilde r_\ast}^{(\mu_{z_\ast;r_\ast})}(z_\ast)
	\subset 
	Q_{\eta\widetilde r_\ast}^{(\mu_{z;\eta\widetilde r_\ast})}(z).
\end{equation}
For the proof of \eqref{inclusion-eta} a distinction must be made between the cases $p\le 2$ and $p\ge2$. We first consider exponents $\max\{\frac{2n}{n+2},1\}<p\le 2$. 
Since $\widetilde r_\ast\ge r_\ast$ and $|x_\ast-x|<4r+4r_\ast\le 12r_\ast$, we know 
$B_{\widetilde r_\ast}(x_\ast)\subset B_{\eta\widetilde r_\ast}(x)$. 
Moreover, due to \eqref{mu-ast} we may conclude that
\begin{align*}
    \mu_{z_\ast;r_\ast}^{p-2}\widetilde r_\ast^{p} + 
    |t_\ast-t|
    &\le
    \mu_{z_\ast;r_\ast}^{p-2}\widetilde r_\ast^{p} +
    \mu_{z;r}^{p-2}(4r)^{p} +
    \mu_{z_\ast;r_\ast}^{p-2}(4 r_\ast)^{p}
    \\
    &\le
    (1+4^p+8^p)
    \mu_{z;\eta \widetilde r_\ast}^{p-2}
    \widetilde r_\ast^{p}
    \le
   \mu_{z;\eta \widetilde r_\ast}^{p-2}
   (\eta\widetilde r_\ast)^{p},
\end{align*}
and this immediately implies the inclusion
$$
	\Lambda_{\widetilde r_\ast}^{(\mu_{z_\ast;r_\ast})}(t_\ast)
	\subset 
	\Lambda_{\eta\widetilde r_\ast}^{(\mu_{z;\eta\widetilde r_\ast})}(t),
$$
so that \eqref{inclusion-eta} is proven for exponents $\max\{\frac{2n}{n+2},1\}<p\le 2$. 
Otherwise, if $2\le p<\frac{2n}{(n-2)_+}$, we have $|t-t_\ast|<(4r)^p+(4r_\ast)^p\le (12r_\ast)^p$ 
and hence 
$\Lambda_{\tilde r_\ast}(t_\ast)\subset \Lambda_{\eta\tilde r_\ast}(t)$.  
Furthermore, \eqref{mu-ast} yields
\begin{align*}
	(\mu_{z_\ast;r_\ast})^{\frac{2-p}{p}} \widetilde r_\ast +
	|x-x_\ast|
	&\le 
	(\mu_{z_\ast;r_\ast})^{\frac{2-p}{p}} \widetilde r_\ast +
	(\mu_{z;r})^{\frac{2-p}{p}}4r +
	(\mu_{z_\ast;r_\ast})^{\frac{2-p}{p}}4r_\ast \\
    &\le
    (\mu_{z;\eta \widetilde r_\ast})^{\frac{2-p}{p}}
    \eta\widetilde r_\ast,
\end{align*}
which implies the inclusion
$$
	B_{\widetilde r_\ast}^{(\mu_{z_\ast;r_\ast})}(x_\ast)
	\subset 
	B_{\eta\widetilde r_\ast}^{(\mu_{z;\eta\tilde r_\ast})}(x).
$$
This establishes the claim \eqref{inclusion-eta} also for the remaining case $2\le p<\frac{2n}{(n-2)_+}$. Now we can finish the proof of \eqref{control-mu-1}. Due to \eqref{control-mu-2}, \eqref{inclusion-eta}, \eqref{sub-intrinsic} applied with $\rho=s=\eta\tilde r_\ast$, and \eqref{mu-ast}, we obtain
\begin{align*}
    \mu_{z_\ast;r_\ast}^{p-\p}
	&\le
	\frac{\eta^{p}}{\lambda^p|Q_{\widetilde r_\ast}|}
	\iint_{Q_{\eta\widetilde r_\ast}^{(\mu_{z;\eta\widetilde r_\ast})}(z)}
	\frac{\abs{u}^{p}}{(\eta\widetilde r_\ast)^p} \dx\dt 
	\le
	\eta^{n+2p} \mu_{z;\eta\widetilde r_\ast}^{p-\p}
	\le
	\eta^{n+2p} \mu_{z;r}^{p-\p} ,
\end{align*}
so that 
\begin{align*}
	\mu_{z_\ast;r_\ast}
	\le
	\eta^{\frac{n+2p}{p-\p}}\,\mu_{z;r}.
\end{align*}
This finishes the proof of \eqref{control-mu-1}. 

It remains to show the inclusion
\begin{equation}\label{inclusion}
	Q =Q_{4r}^{(\mu_{z;r})}(z)\
	\subset 
	\widehat Q^\ast
	\equiv
	Q_{\hat c r_\ast}^{(\mu_{z_\ast;r_\ast})}(z_\ast)
\end{equation}
for a constant $\hat c=\hat c(n,p)\ge 20$.
If $\max\{\frac{2n}{n+2},1\}<p\le 2$ we get with \eqref{control-mu-1} that
\begin{align*}
	\mu_{z;r}^{p-2}(4r)^{p} + |t-t_\ast|
	&\le
    2\mu_{z;r}^{p-2}(4r)^{p}
    +
    \mu_{z_\ast;r_\ast}^{p-2}(4r_\ast)^{p}\\
	&\le
	\Big[ 2^{p+1} (4\eta)^{\frac{n+2p}{p-\p}(2-p)} + 1\Big]
    \mu_{z_\ast;r_\ast}^{p-2}(4r_\ast)^{p} \\
	&\le
	\mu_{z_\ast;r_\ast}^{p-2}
	(\hat c r_\ast)^{p},
\end{align*}
where $\hat c=\hat c(n,p)$ is chosen suitably. This proves that
$\Lambda_{4r}^{(\mu_{z;r})}(t)\subset \Lambda_{\hat c r_\ast}^{(\mu_{z_\ast;r_\ast})}(t_\ast)$.
Moreover, if we choose $\hat c\ge 20$ we have the inclusion $B_{4r}(x)\subset B_{\hat c r_\ast}(x_\ast)$. This implies \eqref{inclusion}. In the case $2\le p<\frac{2n}{(n-2)_+}$, inequality \eqref{control-mu-1} shows
\begin{align*}
	(\mu_{z;r})^{\frac{2-p}{p}}4r + 
	|x-x_\ast|
	&\le
    2(\mu_{z;r})^{\frac{2-p}{p}}4r
    +
    (\mu_{z_\ast;r_\ast})^{\frac{2-p}{p}}4r_\ast \\
	&\le
	\Big[ 4 (4\eta)^{\frac{p-2}{p}\cdot\frac{n+2p}{p-\p}} + 1\Big]
    (\mu_{z_\ast;r_\ast})^{\frac{2-p}{p}}4r_\ast \\
	&\le
	(\mu_{z_\ast;r_\ast})^{\frac{2-p}{p}}
	\hat c r_\ast,
\end{align*}
for a suitable constant $\hat c=\hat c(n,p)$, from which we deduce
$B_{4r}^{(\mu_{z;r})}(x)\subset B_{\hat c r_\ast}^{(\mu_{z_\ast;r_\ast})}(x_\ast)$.
Moreover, if we choose $\hat c\ge 20$ we have the inclusion $\Lambda_{4r}(t)\subset \Lambda_{\hat c r_\ast}(t_\ast)$. Again this implies \eqref{inclusion}. 
In any case we have thus established the claim \eqref{covering}. This completes the proof of the Vitali type covering property.  
\end{proof}

\subsection{Stopping time argument}
We now let
\begin{equation}\label{second-lambda-0}
	\lambda_o
	:=
	1+\Bigg[\biint_{Q_{4R}} 
	\bigg[\frac{\abs{u}^{p}}{(4R)^p} + |Du|^p + |F|^p\bigg] \d x\d t
	\Bigg]^{\frac{1}{p}},
\end{equation}
so that $\lambda_o$ satisfies the previously demanded requirement \eqref{first-lambda-0}.
For $\lambda\ge\lambda_o$ and $r\in(0,2R]$, we define the superlevel set
of  $|Du|$ by 
$$
	\boldsymbol E(r,\lambda)
	:=
	\Big\{z\in Q_{r}: 
	\mbox{$z$ is a Lebesgue point of $|Du|$ and 
	$|Du|(z) > \lambda$}\Big\}.
$$
Here, we mean Lebesgue points of $|Du|$ with respect to the system of cylinders constructed in Section~\ref{sec:cylinders}. 
For radii $R\le R_1<R_2\le 2R$, we consider the concentric parabolic cylinders
$Q_R\subseteq Q_{R_1} \subset Q_{R_2}\subseteq Q_{2R}$.
We fix $z_o\in \boldsymbol E(R_1,\lambda)$ and write $\mu_s\equiv \mu_{z_o;s}^{(\lambda)}$ for $s\in(0,R]$ throughout this section. By Lebesgue's Differentiation Theorem, cf. \cite[\S 2.9.1]{Federer}, we have
\begin{equation}\label{larger-lambda}
	\liminf_{s\downarrow 0} 
	\biint_{Q_{s}^{(\mu_{s})}(z_o)} 
	\big[|Du|^p + |F|^{p}\big] 
	\d x\d t
	\ge
	|Du|^p(z_o)
	>
	\lambda^p.
\end{equation}
By $\hat c=\hat c(n,p)$ we denote the constant from the Vitali type
covering Lemma \ref{lem:vitali}. From now on, we consider values of $\lambda$ satisfying 
\begin{equation}\label{choice_lambda}
	\lambda
	>
	B\lambda_o,
	\qquad\mbox{where }
	\quad
	B
	:=
	\Big(\frac{4\hat c R}{R_2-R_1}\Big)^{\frac{n+2}{p-\p}}
	>1.
\end{equation} 
For $s$ with
\begin{align}\label{radius-s}
	\frac{R_2-R_1}{\hat c}\le s\le R
\end{align}
we have, due to the definition of $\lambda_o$ in \eqref{second-lambda-0}, \eqref{bound-theta-2} and \eqref{radius-s} that
\begin{align*}
	\biint_{Q_{s}^{(\mu_s)}(z_o)} 
	\big[|Du|^p + |F|^p\big] 
	\d x\d t
	&\le
	\frac{|Q_{4R}|}{\big|Q_{s}^{(\mu_s)}\big|}
	\biint_{Q_{4R}} \big[|Du|^p + |F|^p\big] \d x\d t \\
	&\le
	\frac{|Q_{4R}|}{|Q_{s}|}\mu_s^{\p-(\pf-2)}
	\lambda_o^{p} 
	\le
	\Big(\frac{4R}{s}\Big)^{\frac{p(n+2)}{p-\p}} 
	\lambda_o^{p} \\
	&\le
	\Big(\frac{4\hat c R}{R_2-R_1}\Big)^{\frac{p(n+2)}{p-\p}} 
	\lambda_o^{p} 
	=
	(B\lambda_o)^{p} 
	<
	\lambda^{p}.
\end{align*}
On the other hand, due to \eqref{larger-lambda} we find a sufficiently small radius $0 < s < \tfrac{R_2-R_1}{\hat c}$ such that the integral in \eqref{larger-lambda}
possesses a value larger
than $\lambda^{p}$. By the
continuity of $\rho\mapsto\mu_\rho$  and the absolute continuity of
the integral,
there exists a maximal radius $0<\rho_{z_o} < \tfrac{R_2-R_1}{\hat c}$ such that 
\begin{align}\label{=lambda}
	\biint_{Q_{\rho_{z_o}}^{(\mu_{\rho_{z_o}})}(z_o)} 
	\big[|Du|^p + |F|^p\big] \d x\d t
	=
	\lambda^{p}.
\end{align}
By the maximality of $\rho_{z_o}$ we know that 
\begin{align*}
	\biint_{Q_{s}^{(\mu_{s})}(z_o)} 
	\big[|Du|^p + |F|^p\big] \d x\d t
	<
	\lambda^{p}
	\qquad
	\mbox{for any $\rho_{z_o}<s\le R$.}
\end{align*}
Moreover, due to the monotonicity of $\rho\mapsto\mu_{\rho}$ and \eqref{bound-theta} we have
$$
	\mu_{s}
	\le 
	\mu_{\rho_{z_o}}
	\le 
	\Big(\frac{s}{\rho_{z_o}}\Big)^{\frac{n+2p}{p-\p}}\mu_{s},
$$
so that 
\begin{align}\label{<lambda}
	\biint_{Q_{s}^{(\mu_{\rho_{z_o}})}(z_o)} 
	\big[|Du|^p + |F|^p\big]\dx\dt
	&\le
	\Big(\frac{\mu_{\rho_{z_o}}}{\mu_s}\Big)^{\p-(\pf-2)}
	\biint_{Q_{s}^{(\mu_{s})}(z_o)} 
	\big[|Du|^p + |F|^p\big] \dx\dt \nonumber\\
	&<
	\Big(\frac{s}{\rho_{z_o}}\Big)^{\frac{(n+2p)(\p-(\pf-2))}{p-\p}}\,
	\lambda^{p},
\end{align}
for any $\rho_{z_o}<s\le R$. 
Finally, since $R_1^{p}+(R_2-R_1)^{p}\le R_2^{p}$ we have that $Q_{\hat c\rho_{z_o}}^{(\mu_{\rho_{z_o}})}(z_o)\subset Q_{\hat c\rho_{z_o}}(z_o)\subset Q_{R_2}$. 

\subsection{A Reverse H\"older Inequality}
As before, we consider $z_o\in \boldsymbol E(R_1,\lambda)$ with $\lambda$ as in \eqref{choice_lambda}.
Since $\lambda$ and $z_o$ are fixed, we once again use the abbreviation $\mu_{\rho_{z_o}}
:= \mu_{z_o;\rho_{z_o}}^{(\lambda)}$. We keep in mind that
by construction  $0<\rho_{z_o}< \frac{R_2-R_1}{\hat c}$. According to \eqref{rho-tilde} 
we construct  $\widetilde\rho_{z_o}\in[\rho_{z_o},R]$ and recall that,
at least in the case $\widetilde\rho_{z_o}<R$,
the cylinder $Q_{\widetilde\rho_{z_o}}^{(\mu_{\rho_{z_o}})}(z_o)$ is $\mu$-intrinsic, while $Q_{\rho_{z_o}}^{(\mu_{\rho_{z_o}})}(z_o)$ is possibly only $\mu$-sub-intrinsic.
By construction we have  $\mu_s=\mu_{\rho_{z_o}}$ for any $s\in [\rho_{z_o}, \widetilde\rho_{z_o}]$. In particular, $\mu_{\widetilde\rho_{z_o}}=\mu_{\rho_{z_o}}$. Our aim now is to apply Proposition \ref{prop:revhoelder-int} on the cylinder $Q_{2\rho_{z_o}}^{(\mu_{\rho_{z_o}})}(z_o)$. To this aim we have to verify that hypotheses \eqref{subintrinsic-mu} and \eqref{superintrinsic-mu} are fulfilled on this cylinder. 
From \eqref{<lambda} applied with $s=4\rho_{z_o}$ and \eqref{=lambda} we first observe that
\begin{align}\label{lambda-intr}
	c^{-1}\biint_{Q_{4\rho_{z_o}}^{(\mu_{\rho_{z_o}})}(z_o)}
	\big[|Du|^p + |F|^p\big] \dx\dt
	&\le 
	\lambda^p
	=
	\biint_{Q_{\rho_{z_o}}^{(\mu_{\rho_{z_o}})}(z_o)}
	\big[|Du|^p + |F|^p\big] \dx\dt \nonumber\\
	&\le
	2^{n+p}\biint_{Q_{2\rho_{z_o}}^{(\mu_{\rho_{z_o}})}(z_o)}
	\big[|Du|^p + |F|^p\big] \dx\dt,
\end{align}
for a constant $c=c(n,p)>1$. 
Together with \eqref{sub-intrinsic} applied with $s=4\rho_{z_o}$ this shows
\begin{align*}
	\frac{\displaystyle{
	\biint_{Q_{4\rho_{z_o}}^{(\mu_{\rho_{z_o}})}(z_o)}
	\frac{|u|^p}{{\mu_{\rho_{z_o}}^{2-\pf}(4\rho_{z_o}})^p} \dx\dt}}
	{\displaystyle{
	\biint_{Q_{2\rho_{z_o}}^{(\mu_{\rho_{z_o}})}(z_o)}
	\big[|Du|^p + |F|^p\big] \d x\d t}}
	\le
	2^{n+p}\mu_{\rho_{z_o}}^p,
\end{align*}
ensuring that \eqref{subintrinsic-mu} is satisfied for the cylinder $Q_{2\rho_{z_o}}^{(\mu_{\rho_{z_o}})}(z_o)$ with $K=2^{n+p}$. We now turn our attention to hypothesis \eqref{superintrinsic-mu}. If $\widetilde\rho_{z_o}\le 2\rho_{z_o}$ we use the fact that $\mu_{\widetilde\rho_{z_o}}=\mu_{\rho_{z_o}}$ and inequality \eqref{lambda-intr} to infer that
\begin{align*}
	\mu_{\rho_{z_o}}^p
	=
	\frac{1}{\lambda^p}
	\biint_{Q_{\widetilde\rho_{z_o}}^{(\mu_{\rho_{z_o}})}(z_o)}
	\frac{|u|^p}{\mu_{\rho_{z_o}}^{2-\pf}\widetilde\rho_{z_o}^p} \dx\dt
	\le
	c\,
	\frac{\displaystyle{
	\biint_{Q_{2\rho_{z_o}}^{(\mu_{\rho_{z_o}})}(z_o)}
	\frac{|u|^p}{\mu_{\rho_{z_o}}^{2-\pf}(2\rho_{z_o})^p} \dx\dt}}
	{\displaystyle{
	\biint_{Q_{4\rho_{z_o}}^{(\mu_{\rho_{z_o}})}(z_o)}
	\big[|Du|^p + |F|^p\big] \d x\d t}},
\end{align*}
for a constant $c=c(n,p)$. 
This shows that $Q_{2\rho_{z_o}}^{(\mu_{\rho_{z_o}})}(z_o)$ satisfies \eqref{superintrinsic-mu}$_1$ with $K=c(n,p)$. 
It remains to consider the case $\widetilde\rho_{z_o}> 2\rho_{z_o}$. 
If $\mu_{\rho_{z_o}}=1$, then \eqref{superintrinsic-mu}$_2$ is satisfied with $K=1$. 
If $\mu_{\rho_{z_o}}>1$, then by construction $Q_{\widetilde\rho_{z_o}}^{(\mu_{\rho_{z_o}})}(z_o)$ is intrinsic. 
Using in turn Lemma~\ref{lem:alphalemma}, inequality \eqref{sub-intrinsic} with $(\rho,s)$ replaced by $(\rho_{z_o}, \frac12\widetilde\rho_{z_o})$ (note that this is possible since $\frac12 \widetilde\rho_{z_o}\ge \rho_{z_o}$), Lemma \ref{lem:poin-1}, respectively Lemma~\ref{lem:poin-2} (for $q=p$) and \eqref{<lambda} (applied with $s=\widetilde\rho_{z_o}\in (\rho_{z_o}, R]$) we obtain  \begin{align*}
	\mu_{\rho_{z_o}}\lambda
	&=
	\bigg[\biint_{Q_{\widetilde\rho_{z_o}}^{(\mu_{\rho_{z_o}})}(z_o)} 
	\frac{\abs{u}^{p}}{\mu_{\rho_{z_o}}^{2-\pf}\widetilde\rho_{z_o}^p} \dx\dt\bigg]^\frac1{p} 
	\\
	&\le 
	\bigg[\biint_{Q_{\widetilde\rho_{z_o}}^{(\mu_{\rho_{z_o}})}(z_o)} 
	\frac{\big|u - (u)_{z_o;\widetilde\rho_{z_o}/2}^{(\mu_{\rho_{z_o}})}\big|^{p}}{\mu_{\rho_{z_o}}^{2-\pf}\widetilde\rho_{z_o}^p} \dx\dt
	\bigg]^\frac1{p} 
	+
	\frac{\babs{(u)_{z_o;\widetilde\rho_{z_o}/2}^{(\mu_{\rho_{z_o}})}}}
	{\mu_{\rho_{z_o}}^{\frac{2-\pf}{p}}\widetilde\rho_{z_o}}
	 \\
	&\le 
	c\bigg[\biint_{Q_{\widetilde\rho_{z_o}}^{(\mu_{\rho_{z_o}})}(z_o)} 
	\frac{\big|u {-} (u)_{z_o;\widetilde\rho_{z_o}}^{(\mu_{\rho_{z_o}})}\big|^{p}}
	{\mu_{\rho_{z_o}}^{2-\pf}\widetilde\rho_{z_o}^p} \dx\dt\bigg]^\frac1{p} 
	+
	\bigg[\biint_{Q_{\widetilde\rho_{z_o}/2}^{(\mu_{\rho_{z_o}})}(z_o)} 
	\frac{
	\abs{u}^{p}}{\mu_{\rho_{z_o}}^{2-\pf}\widetilde\rho_{z_o}^p} \dx\dt\bigg]^\frac1{p} \\
	&\le 
	c\bigg[ \biint_{Q_{\widetilde\rho_{z_o}}^{(\mu_{\rho_{z_o}})}(z_o)} 
	\big[|Du|^p + |F|^p \big] 
	\dx\dt\bigg]^\frac1{p} +
	\tfrac12 \mu_{\rho_{z_o}} \lambda\\
	&\le 
	c\, \lambda +
	\tfrac12 \mu_{\rho_{z_o}}\lambda,
\end{align*}
with $c=c(n,p,L)$. 
After re-absorbing $\tfrac12\mu_{\rho_{z_o}}\lambda$ into the left-hand side, we find that
$\mu_{\rho_{z_o}}\le c(n,p,L)$.
This ensures that \eqref{superintrinsic-mu}$_2$ is satisfied with $K=c(n,p,L)$. 
Therefore, we are allowed to apply Proposition \ref{prop:revhoelder-int} on the cylinder $Q_{2\rho_{z_o}}^{(\mu_{\rho_{z_o}})}(z_o)$ with a constant $K=K(n,p,L)$ and thereby obtain the reverse H\"older inequality
\begin{align}\label{rev-hoelder}
	\biint_{Q_{2\rho_{z_o}}^{(\mu_{\rho_{z_o}})}(z_o)} & |Du|^p \dx\dt \nonumber\\
	&\le
	c\bigg[\biint_{Q_{4\rho_{z_o}}^{(\mu_{\rho_{z_o}})}(z_o)} 
	|Du|^{q} \dx\dt \bigg]^{\frac{p}{q}} +
	c\, \biint_{Q_{4\rho_{z_o}}^{(\mu_{\rho_{z_o}})}(z_o)} |F|^p \dx\dt,
\end{align}
with $q:=\max\big\{\frac{2n}{n+2},\frac{np}{n+2},1,p-1\big\}<p$ and $c=c(n,p,\nu, L)$.

\subsection{Estimates on level sets}
We summarize what we have shown so far. If $\lambda$ satisfies \eqref{choice_lambda}, then for any $z_o\in E(R_1,\lambda)$ there exists a  cylinder $Q_{\rho_{z_o}}^{(\mu_{z_o;\rho_{z_o}}^{(\lambda)})}(z_o)$ such that the $\hat c$-times enlarged cylinder $Q_{\hat c\rho_{z_o}}^{(\mu_{z_o;\rho_{z_o}}^{(\lambda)})}(z_o)$ is still contained in $ Q_{R_2}$, and such that  \eqref{=lambda}, \eqref{<lambda} and \eqref{rev-hoelder} hold on this cylinder. 
As before, we abbreviate $\mu_{\rho_{z_o}}\equiv \mu_{z_o;\rho_{z_o}}^{(\lambda)}$.
Moreover, we define the superlevel set of the inhomogeneity $|F|$ by 
$$
	\boldsymbol F(r,\lambda)
	:=
	\Big\{z\in Q_{r}: 
	\mbox{$z$ is a Lebesgue point of $|F|$ and 
	$|F|(z)>\lambda$}\Big\}
$$
and let $\eta\in (0,1]$ to be specified later. 
Due to \eqref{=lambda} and \eqref{rev-hoelder} we have
\begin{align*}
	\lambda^{p}
	&=
	\biint_{Q_{\rho_{z_o}}^{(\mu_{\rho_{z_o}})}(z_o)} 
	\big[|Du|^p + |F|^p\big] \d x\d t \\
	&\le
	c\,\bigg[\biint_{Q_{4\rho_{z_o}}^{(\mu_{\rho_{z_o}})}(z_o)} 
	|Du|^{q} \dx\dt \bigg]^{\frac{p}{q}} +
	c\, \biint_{Q_{4\rho_{z_o}}^{(\mu_{\rho_{z_o}})}(z_o)} |F|^{p} \dx\dt\\
	&\le
	c\,\eta^{p}\lambda^{p} +
	c\,\Bigg[
	\frac{1}{\big|Q_{4\rho_{z_o}}^{(\mu_{\rho_{z_o}})}(z_o)\big|}
	\iint_{Q_{4\rho_{z_o}}^{(\mu_{\rho_{z_o}})}(z_o)\cap \boldsymbol E(R_2,\eta\lambda)} 
	|Du|^{q} \dx\dt \Bigg]^{\frac{p}{q}} \\
	&\quad+
	\frac{c}{\big|Q_{4\rho_{z_o}}^{(\mu_{\rho_{z_o}})}(z_o)\big|}
	\iint_{Q_{4\rho_{z_o}}^{(\mu_{\rho_{z_o}})}(z_o)\cap \boldsymbol F(R_2,\eta\lambda)} 
	|F|^{p} \dx\dt\\
	&\le
	c\, \eta^{p}\lambda^{p} +
	\frac{c}{\big|Q_{4\rho_{z_o}}^{(\mu_{\rho_{z_o}})}(z_o)\big|}
	\iint_{Q_{4\rho_{z_o}}^{(\mu_{\rho_{z_o}})}(z_o)\cap \boldsymbol E(R_2,\eta\lambda)} 
	|Du|^{q} \dx\dt\cdot 
	\mathrm I \\
	&\quad+
	\frac{c}{\big|Q_{4\rho_{z_o}}^{(\mu_{\rho_{z_o}})}(z_o)\big|}
	\iint_{Q_{4\rho_{z_o}}^{(\mu_{\rho_{z_o}})}(z_o)\cap \boldsymbol F(R_2,\eta\lambda)} 
	|F|^{p} \dx\dt,
\end{align*}
with $c=c(n,p,\nu ,L)$ and
$$
	\mathrm I:= \bigg[\biint_{Q_{4\rho_{z_o}}^{(\mu_{\rho_{z_o}})}(z_o)} 
	|Du|^{q} \dx\dt \bigg]^{\frac{p}{q}-1}.
$$
In view of H\"older's inequality and \eqref{<lambda} we find that
\begin{align*}
	\mathrm I
	\le
	\bigg[\biint_{Q_{4\rho_{z_o}}^{(\mu_{\rho_{z_o}})}(z_o)} 
	|Du|^{p} \dx\dt \bigg]^{1-\frac{q}{p}}
	\le
	c\,\lambda^{p-q}.
\end{align*}
We insert this inequality above. Then, we choose
$\eta=(\frac{1}{2c})^{\frac1p}$ and re-absorb $\frac12\lambda^{p}$ into the left-hand side. 
Multiplying the result by $\big|Q_{4\rho_{z_o}}^{(\mu_{\rho_{z_o}})}(z_o)\big|$ yields
\begin{align*}
	\lambda^{p}\Big|Q_{4\rho_{z_o}}^{(\mu_{\rho_{z_o}})}(z_o)\Big|
	&\le
	c\iint_{Q_{4\rho_{z_o}}^{(\mu_{\rho_{z_o}})}(z_o)\cap\mathbf E(R_2,\eta\lambda)} 
	\lambda^{p-q}|Du|^{q} \dx\dt \\
	&\quad+
	c\,
	\iint_{Q_{4\rho_{z_o}}^{(\mu_{\rho_{z_o}})}(z_o)\cap \mathbf F(R_2,\eta\lambda)} 
	 |F|^{p} \dx\dt,
\end{align*}
again with  $c=c(n,p,\nu,L)$. Now,  \eqref{<lambda} for the choice
$s=\hat c\rho_{z_o}$ allows us
to estimate  $\lambda^{p}$ from below. In this way, we deduce 
\begin{align}\label{level-est}
	\iint_{Q_{\hat c\rho_{z_o}}^{(\mu_{\rho_{z_o}})}(z_o)} 
	|Du|^p \dx\dt 
	&\le
    c
    \iint_{Q_{4\rho_{z_o}}^{(\mu_{\rho_{z_o}})}(z_o)\cap \boldsymbol E(R_2,\eta\lambda)} 
	\lambda^{p-q}|Du|^{q} \dx\dt \nonumber\\
	&\quad +
	c
	\iint_{Q_{4\rho_{z_o}}^{(\mu_{\rho_{z_o}})}(z_o)\cap \boldsymbol F(R_2,\eta\lambda)} 
	|F|^{p} \dx\dt,
\end{align}
where $c=c(n,p,\nu,L)$. Since $z_o\in \boldsymbol E(R_1,\lambda)$ was arbitrary, we have thus shown that for any $\lambda>B\lambda_o$ the associated super-level set $\boldsymbol E(R_1,\lambda)$ is covered by a  family
$$
	\mathcal F
	\equiv
	\Big\{Q_{4\rho_{z_o}}^{(\mu_{z_o;\rho_{z_o}})}(z_o)\Big\}
$$
of parabolic cylinders 
with center $z_o\in \boldsymbol E(R_1,\lambda)$ which are contained in $Q_{R_2}$, and such that
\eqref{level-est} holds true on each of these cylinders. Recall, since $\lambda$ is fixed we again write $\mu_{z_o;\rho_{z_o}}\equiv\mu_{z_o;\rho_{z_o}}^{(\lambda)}$. The Vitali type covering Lemma~\ref{lem:vitali} now ensures that there exists a countable subfamily
$$
	\Big\{
	Q_{4\rho_{z_i}}^{(\mu_{z_i;\rho_{z_i}})}(z_i)
	\Big\}_{i\in\N}
	\subset \mathcal F
$$
of pairwise disjoint  cylinders, such that the $\frac14\hat c $-times enlarged cylinders $Q_{\hat c\rho_{z_i}}^{(\mu_{z_i;\rho_{z_i}})}(z_i)$ cover the super-level set $\boldsymbol E(R_1,\lambda)$ and are still contained in $Q_{R_2}$. More precisely, we have 
$$
	\boldsymbol E(R_1,\lambda)
	\subset 
	\bigcup_{i=1}^\infty 
	Q_{\hat c\rho_{z_i}}^{(\mu_{z_i;\rho_{z_i}})}(z_i)
	\subset
	Q_{R_2}.
$$
Since the cylinders $Q_{4\rho_{z_i}}^{(\mu_{z_i;\rho_{z_i}})}(z_i)$ are pairwise disjoint we obtain with \eqref{level-est} that 
\begin{align*}
	\iint_{\boldsymbol E(R_1,\lambda)} 
	|Du|^p \d x\d t
	&\le
	\sum_{i=1}^\infty
	\iint_{Q_{\hat c\rho_{z_i}}^{(\mu_{z_i;\rho_{z_i}}
)}(z_i)} 
	|Du|^p \d x\d t \nonumber\\
	&\le
	c\sum_{i=1}^\infty
	\iint_{Q_{4\rho_{z_i}}^{(\mu_{z_i;\rho_{z_i}})}(z_i)\cap \boldsymbol E(R_2,\eta\lambda)} 
	\lambda^{p-q}|Du|^{q} \dx\dt \\
	&\phantom{\le\,}
	+
	c\sum_{i=1}^\infty
	\iint_{Q_{4\rho_{z_i}}^{(\mu_{z_i;\rho_{z_i}})}(z_i)\cap \boldsymbol F(R_2,\eta\lambda)} 
	|F|^{p} \dx\dt \\
	&\le
	c\iint_{\boldsymbol E(R_2,\eta\lambda)} 
	\lambda^{p-q} |Du|^{q} \dx\dt +
	c\iint_{\boldsymbol F(R_2,\eta\lambda)} |F|^{p} \dx\dt,
\end{align*}
for a  constant $c=c(n, p, \nu, L)$. 
On $\boldsymbol  E(R_1,\eta\lambda)\setminus \boldsymbol E(R_1,\lambda)$ we have  $|Du|\le\lambda$, so that
\begin{align*}
	\iint_{\boldsymbol  E(R_1,\eta\lambda)\setminus \mathbf E(R_1,\lambda)} 
	|Du|^p \d x\d t
	&\le
	\iint_{\boldsymbol E(R_1,\eta\lambda)\setminus \boldsymbol E(R_1,\lambda)} 
	\lambda^{p-q}|Du|^{q} \dx\dt \\
	&\le
	\iint_{\boldsymbol E(R_2,\eta\lambda)} 
	\lambda^{p-q}|Du|^{q} \dx\dt.
\end{align*}
Combining this with the second last inequality yields
\begin{align*}
	\iint_{\boldsymbol E(R_1,\eta\lambda)} 
	|Du|^p \dx\dt 
	\le
	c\iint_{\boldsymbol E(R_2,\eta\lambda)} 
	\lambda^{p-q}|Du|^{q} \dx\dt +
	c \iint_{\boldsymbol F(R_2,\eta\lambda)} |F|^p \dx\dt.
\end{align*}
We now replace $\eta\lambda$ by $\lambda$ and recall that $\eta
=\eta (n,p,\nu,L)<1$.
With this replacement we obtain for any $\lambda> \eta B\lambda_o=:\lambda_1$ that
\begin{align}\label{pre-1}
	\iint_{\boldsymbol E(R_1,\lambda)} 
	|Du|^p \d x\d t 
	&\le
	c \iint_{\boldsymbol E(R_2,\lambda)} 
	\Big(\frac{\lambda}{\eta}\Big)^{p-q}
	|Du|^{q} \dx\dt +
	c \iint_{\boldsymbol F(R_2,\lambda)} |F|^{p} \dx\dt\nonumber\\
	&=
	c\iint_{\boldsymbol  E(R_2,\lambda)} 
	\lambda^{p-q}|Du|^{q} \dx\dt +
	c \iint_{\boldsymbol  F(R_2,\lambda)} |F|^{p} \dx\dt
\end{align}
holds true with a constant $c=c(n,p,\nu ,L)$. This is the
{\it reverse H\"older inequality on super-level sets} we are looking for.

\subsection{Proof of the gradient estimate}\label{sec:vitali}
At this point the quantitative higher integrability estimate follows in a standard way from the reverse H\"older inequality on super-level sets
by  multiplying \eqref{pre-1} by $\lambda^{\epsilon p-1}$ and then
integrating with respect to $\lambda$. For the sake of completeness we nevertheless provide the details. 
The just described procedure would lead
on the left to an integral of $|Du|^{p(1+\eps )}$ on 
$Q_{R_1}$, while on the right the same integral appears
with factor $\frac12$ and $Q_{R_2}$ as domain of integration. If both integrals are finite the one on the right could be re-absorbed in view of Lemma~\ref{lem:tech}. 
However, it is not clear in advance
that these integrals are finite. For this reason we use a truncation argument in order to avoid powers of $|Du|$ that are larger than $p$. The rigorous argument is as follows:
For  $k> \lambda_1$ we define 
the {\it truncation} of $|Du|$ by
$$
	|Du|_k
	:=
	\min\{|Du|, k\},
$$
and for $r\in(0,2R]$ the corresponding super-level set by
$$
	\boldsymbol E_k(r,\lambda)
	:=
	\big\{z\in Q_r: |Du|_k>\lambda\big\}.
$$
Note that $|Du|_k\le |Du|$ a.e., as well as $\boldsymbol E_k(r,\lambda)=\emptyset$ for $k\le\lambda$ and $\boldsymbol  E_k(r,\lambda)=\boldsymbol E(r,\lambda)$ for $k>\lambda$. Therefore, \eqref{pre-1} implies
\begin{align*}
	\iint_{\boldsymbol  E_k(R_1,\lambda)} 
	|Du|_k^{p-q}|Du|^{q} \d x\d t
	\le
	c\iint_{\boldsymbol  E_k(R_2,\lambda)} 
	\lambda^{p-q}|Du|^{q} \dx\dt +
	c \iint_{\boldsymbol F(R_2,\lambda)} |F|^{p} \dx\dt.
\end{align*}
We now multiply this inequality by $\lambda^{\epsilon p-1}$ with some $\epsilon\in (0,1]$ to be chosen later. Integrating the result with respect to $\lambda$ over the interval $(\lambda_1,\infty)$ leads to
\begin{align*}
	\int_{\lambda_1}^\infty \lambda^{\epsilon p-1} &
	\bigg[\iint_{\boldsymbol  E_k(R_1,\lambda)} 
	|Du|_k^{p-q}
	|Du|^{q} \d x\d t \bigg]\d\lambda
	\\
	&\le
	c\int_{\lambda_1}^\infty \lambda^{p-q+\epsilon p-1} 
	\bigg[\iint_{\boldsymbol E_k(R_2,\lambda)} 
	|Du|^{q} \dx\dt\bigg]\d\lambda \\
	&\quad +
	c \int_{\lambda_1}^\infty \lambda^{\epsilon p-1}
	\bigg[\iint_{\boldsymbol  F(R_2,\lambda)} |F|^{p} \dx\dt\bigg]\d\lambda.
\end{align*}
The idea now is to exchange the order of integration in each of the integrals
by an application of Fubini's theorem. For the integral on the left-hand side Fubini's theorem shows
\begin{align*}
	\int_{\lambda_1}^\infty &\lambda^{\epsilon p-1} 
	\iint_{\boldsymbol E_k(R_1,\lambda)} 
	|Du|_k^{p-q}|Du|^{q} \,\dx\dt \d\lambda \\
	&=
	\iint_{\boldsymbol  E_k(R_1,\lambda_1)}
	|Du|_k^{p-q}|Du|^{q}
	\int_{\lambda_1}^{|Du|_k} 
	\lambda^{\epsilon p-1} \,\d\lambda
	 \d x\d t \\
	&=
	\frac{1}{\epsilon p} \iint_{\boldsymbol  E_k(R_1,\lambda_1)}
	\Big[|Du|_k^{p-q+\epsilon p}|Du|^q
	-
	\lambda_1^{\epsilon p} |Du|_k^{p-q}|Du|^q \Big]
	\d x\d t ,
\end{align*}
while for the first integral on the right we get
\begin{align*}
	\int_{\lambda_1}^\infty &\lambda^{p-q+\epsilon p-1} 
	\iint_{\boldsymbol  E_k(R_2,\lambda)} 
	|Du|^{q} \,\dx\dt\d\lambda \\
	&=
	\iint_{\boldsymbol  E_k(R_2,\lambda_1)} |Du|^{q}
	\int_{\lambda_1}^{|Du|_k} 
	\lambda^{p-q+\epsilon p-1} \,\d\lambda
	\dx\dt \\
	&\le 
	\frac{1}{p-q+\epsilon p} \iint_{\boldsymbol  E_k(R_2,\lambda_1)} 
	|Du|_k^{p-q+\epsilon p} |Du|^{q} \,\dx\dt \\
	&\le 
	\frac{1}{p-q} \iint_{\boldsymbol  E_k(R_2,\lambda_1)} 
	|Du|_k^{p-q+\epsilon p} |Du|^{q} \,\dx\dt .
\end{align*}
Finally, for the second integral on the right we find that
\begin{align*}
	\int_{\lambda_1}^\infty \lambda^{\epsilon p-1}
	\iint_{\boldsymbol  F(R_2,\lambda)} |F|^{p} \,\dx\dt\d\lambda 
	&=
	\iint_{\boldsymbol  F(R_2,\lambda_1)} |F|^{p}
	\int_{\lambda_1}^{|F|} 
	\lambda^{\epsilon p-1} \,\d\lambda
	 \dx\dt \\
	&\le
	\frac{1}{\epsilon p} \iint_{\boldsymbol F(R_2,\lambda_1)}
	|F|^{(1+\epsilon)p} \,\dx\dt \\
	&\le
	\frac{1}{\epsilon p} \iint_{Q_{2R}}
	|F|^{(1+\epsilon)p} \,\dx\dt .
\end{align*}
Inserting the preceding estimates above and multiplying by $\epsilon p$ shows that
\begin{align*}
	 \iint_{\boldsymbol  E_k(R_1,\lambda_1)}&
	|Du|_k^{p-q+\epsilon p}|Du|^{q} \,\dx\dt \\
	&\le
	\lambda_1^{\epsilon p} 
	\iint_{\boldsymbol E_k(R_1,\lambda_1)}
	|Du|_k^{p-q}|Du|^{q} \,\dx\dt \\
	&\phantom{\le\,}+
	\frac{c\, \epsilon p}{p-q} \iint_{\boldsymbol  E_k(R_2,\lambda_1)} 
	|Du|_k^{p-q+\epsilon p}|Du|^{q} \,\dx\dt +
	c\iint_{Q_{2R}} |F|^{(1+\epsilon)p} \,\dx\dt .	
\end{align*}
On the complement $Q_{R_1}\setminus \boldsymbol  E_k(R_1,\lambda_1)$ we have $|Du|_k\le\lambda_1$ and hence
\begin{align*}
	\iint_{Q_{R_1}\setminus \mathbf E_k(R_1,\lambda_1)}&
	|Du|_k^{p-q+\epsilon p}|Du|^{q} \,\dx\dt \\
	&\le
	\lambda_1^{\epsilon p} 
	\iint_{Q_{R_1}\setminus \mathbf E_k(R_1,\lambda_1)}
	|Du|_k^{p-q}|Du|^{q} \,\dx\dt .
\end{align*}
Joining the last two estimates and taking into account that $|Du|_k\le |Du|$, we obtain
\begin{align*}
	 \iint_{Q_{R_1}}
	|Du|_k^{p-q+\epsilon p}|Du|^q \,\dx\dt 
	&\le
	\frac{c_\ast \epsilon p}{p-q} \iint_{Q_{R_2}} 
	|Du|_k^{p-q+\epsilon p}|Du|^{q} \,\dx\dt \\
	&\phantom{\le\ }+
	\lambda_1^{\epsilon p} 
	\iint_{Q_{2R}} |Du|^{p} \,\dx\dt +
	c\iint_{Q_{2R}} |F|^{(1+\epsilon)p} \,\dx\dt,
\end{align*}
where $c_\ast=c_\ast(n,p,\nu,L)\ge 1$. Now, we choose
$$
	0
	<
	\epsilon
	\le
	\epsilon_o
	:=
	\frac{p-q}{2pc_\ast}\,.
$$
Note that $\epsilon_o$ depends only on $n,p,\nu$, and $L$. Furthermore,  $\lambda_1^\epsilon\equiv (\eta B\lambda_o)^\epsilon\le B \lambda_o^\epsilon$, since $B\ge 1$, $\eta<1$ and $\epsilon\le 1$. With this choice the last inequality shows that for each pair of radii $R_1$, $R_2$ with  $R\le R_1<R_2\le 2R$ the estimate
\begin{align*}
	 \iint_{Q_{R_1}}&
	|Du|_k^{p-q+\epsilon p}|Du|^{q} \,\dx\dt \\
	&\le
	\tfrac{1}{2} \iint_{Q_{R_2}} 
	|Du|_k^{p-q+\epsilon p}|Du|^{q} 
	\,\dx\dt \\
	&\phantom{\le\ }+
	c\,\lambda_o^{\epsilon p}
	\bigg(\frac{2R}{R_2-R_1}\bigg)^{\frac{p(n+2)}{p-\p}} 
	\iint_{Q_{2R}} |Du|^{p} \,\dx\dt +
	c\iint_{Q_{2R}} |F|^{(1+\epsilon)p} \,\dx\dt
\end{align*}
holds true. In view of the Iteration Lemma~\ref{lem:tech} we conclude that
\begin{align*}
	\iint_{Q_{R}}
	|Du|_k^{p-q+\epsilon p}|Du|^{q} \,\dx\dt 
	\le
	c\, 
	\lambda_o^{\epsilon p} 
	\iint_{Q_{2R}} |Du|^{p} \,\dx\dt
	+
	c\iint_{Q_{2R}} |F|^{(1+\epsilon)p} \,\dx\dt.
\end{align*}
At this point we use Fatou's lemma to pass to the limit $k\to\infty$ on the left-hand side.
Subsequently we take means on both sides and infer that 
\begin{align*}
	\biint_{Q_{R}} |Du|^{(1+\epsilon)p} \,\dx\dt 
	&\le
	c\, 
	\lambda_o^{\epsilon p} 
	\biint_{Q_{2R}} |Du|^{p} \,\dx\dt +
	c\,\biint_{Q_{2R}} |F|^{(1+\epsilon)p} 
	\,\dx\dt.
\end{align*}
In view of the definition of $\lambda_o$ from \eqref{second-lambda-0} the preceding inequality turns into
\begin{align*}
	\biint_{Q_{R}}& |Du|^{(1+\epsilon)p} \,\dx\dt \\
	&\le
	c \Bigg[1+\biint_{Q_{4R}} 
	\bigg[\frac{\abs{u}^{p}}{(4R)^p} + |Du|^p + |F|^p\bigg] \,\dx\dt
	\Bigg]^{\epsilon}
	\biint_{Q_{2R}} |Du|^{p} \,\dx\dt\\
	&\phantom{\le\,}+
	c\,\biint_{Q_{2R}} |F|^{(1+\epsilon)p} \,\dx\dt \\
	&\le
	c \Bigg[1+\biint_{Q_{4R}} 
	\bigg[\frac{\abs{u}^{p}}{(4R)^p} + |Du|^p\bigg] \,\dx\dt
	\Bigg]^{\epsilon}
	\biint_{Q_{2R}} |Du|^{p} \,\dx\dt\\
	&\phantom{\le\,}+
	c\,\biint_{Q_{2R}} |F|^{(1+\epsilon)p} \,\dx\dt.
\end{align*}
Note that $c=c(n,p,\nu ,L)$.
A straightforward covering argument now yields the claimed quantitative 
estimate. 
This completes the proof of Theorem~\ref{thm:higherint}.

\end{document}